\documentclass[11pt]{article}
\usepackage{verbatim} 
\usepackage{amsmath}
\usepackage{amssymb}
\usepackage{amsfonts}
\usepackage{amsthm}
\usepackage{cite}
\usepackage{graphicx} %,obrázky
\usepackage{pgf,tikz, pgfplots}
\usepackage{tikz-cd} % diagramy
\usepackage{mathrsfs}
\usetikzlibrary{arrows}
\usepackage{color}
\usepackage{enumitem}
\usepackage{latexsym}

\newtheorem{theorem}{Theorem}

\newtheorem{lemma}[theorem]{Lemma}
\newtheorem{proposition}[theorem]{Proposition}

\theoremstyle{definition}

\newtheorem{definition}[theorem]{Definition}
\newtheorem{question}[theorem]{Question}
\newtheorem{claim}{Claim}
\theoremstyle{remark}
\newtheorem*{claimproof}{Proof}

\newcommand{\R}{\mathbb{R}}
\newcommand{\N}{\mathbb{N}}
\newcommand{\Int}{\mathrm{Int}}
\newcommand{\ins}{\mathrm{ins}}
\newcommand{\out}{\mathrm{out}}
\newcommand{\claimend}{{\hfill $\blacksquare$}}

\newcommand{\norm}[1]{\left\lVert#1\right\rVert}

\renewcommand{\restriction}{\mathord{\upharpoonright}}

\begin{document}   
     
\makeatletter
\@namedef{subjclassname@2010}{
  \textup{2010} Mathematics Subject Classification 03E15, 54H05 and 54F15}
\makeatother

\title{The complexity of homeomorphism relations on some classes of compacta with bounded topological dimension}

\author{
Jan Dud\'ak\footnote{https://orcid.org/0000-0003-0627-6641}
\footnote{Supported by the grant SVV-2020-260583}, Benjamin Vejnar\footnote{https://orcid.org/0000-0002-2833-5385} \footnote{
%The author was supported by the grant ******.
This work has been supported by Charles University Research Centre program No.UNCE/SCI/022.
}
\\ Department of Mathematical Analysis
\\ Faculty of Mathematics and Physics, Charles University, Prague, Czechia\\
%\\ Email: vejnar@karlin.mff.cuni.cz
}
\maketitle
\begin{abstract}
We are dealing with the complexity of the homeomorphism equivalence relation on some classess of metrizable compacta from the viewpoint of invariant descriptive set theory.
% Classification results
We prove that the homeomorphism equivalence relation of absolute retracts in the plane is Borel bireducible with the isomorphism equivalence relation of countable graphs.
% Non-classification results
In order to stress the sharpness of this result we prove that neither the homeomorphism relation of locally connected continua in the plane, nor the homeomorphism relation of absolute retracts in $\R^3$ is Borel reducible to the isomorphism relation of countable graphs.
% Classification results
We also improve the recent results of Chang and Gao by constructing a Borel reduction from both the homeomorphism equivalence relation of compact subsets of $\R^n$ and the ambient homeomorphism equivalence relation of compact subsets of $[0,1]^n$ to the homeomorphism equivalence relation of $n$-dimensional continua in $\R^{n+1}$.
\end{abstract}
%\tableofcontents

\section{Introduction}

The task of measuring the complexity of an equivalence relation (ER) on a structure is very complex in itself. In this paper we use the notion of Borel reducibility (see Definitions \ref{DefinitionOfReduction}, \ref{DefinitionOfReducibility} and \ref{DefinitionOfBorelReducibility}) and the results of invariant descriptive set theory to compare the complexities of ERs on Polish and standard Borel spaces. For getting familiar with the field of invariant descriptive set theory we recommend the book by Su Gao \cite{Gao}.

Let us mention several ERs that have become milestones in the theory of Borel reductions (in ascending order with respect to their complexity):

\begin{itemize}[noitemsep]
    \item the equality on an uncountable Polish space (equivalently, on $\R$);
    \item the $S_{\infty}$-universal orbit equivalence relation;
    \item the universal orbit equivalence relation;
    \item the universal analytic equivalence relation.
\end{itemize}

Let us present a few examples of ``real-life" ERs with the same complexity as the four relations above, respectively. It was shown by Gromov that the isometry ER of compact metric spaces is Borel bireducible with the equality of real numbers (see e.g. \cite[Theorem 14.2.1]{Gao}). The isomorphism ER of countable graphs is Borel bireducible with the $S_{\infty}$-universal orbit ER (see \cite[Theorem 13.1.2]{Gao}). Melleray proved in \cite{Melleray} that the isometry ER of separable Banach spaces is Borel bireducible with the universal orbit ER. Ferenczi, Louveau and Rosendal proved in \cite{FerencziLouveauRosendal} that the isomorphism ER of separable Banach spaces is Borel bireducible with the universal analytic ER.

In this paper we study the homeomorphism ER on some subclasses of metrizable compacta. A crucial result by Zielinski (see \cite{Zielinski}) states that the homeomorphism ER of metrizable compacta is Borel bireducible with the universal orbit ER. Chang and Gao have shown in \cite{ChangGao} that the same applies to the homeomorphism ER of metrizable continua. This remains true even if we restrict ourselves to locally connected continua, as was proved by Cie\'{s}la in \cite{Ciesla}. Finally, Krupski and Vejnar proved in \cite{KrupskiVejnar} that the homeomorphism ER of absolute retracts possesses the same complexity as the homeomorphism ER of metrizable compacta.

Now let us turn our attention to homeomorphism ERs which are less complex than the universal orbit ER. The homeomorphism ER of compacta in $\R$ is Borel bireducible with the $S_{\infty}$-universal orbit ER (see \cite[Theorem 4.2]{ChangGaoDim}). On the other hand, it is known (it follows e.g. from \cite[Theorem 4.3]{ChangGaoDim}) that for every $n \in \N$ with $n > 1$ the homeomorphism ER of compacta in $\R^n$ is stricly more complex, but it is not known whether it is as complex as the universal orbit ER. The homeomorphism ER of metrizable rim-finite continua was shown in \cite{KrupskiVejnar} to be Borel bireducible with the $S_{\infty}$-universal orbit ER. However, the homeomorphism ER of metrizable rim-finite compacta is strictly more complex, as was shown in the same paper. The homeomorphism ER of dendrites was shown in \cite{CamerloDarjiMarcone} to be Borel bireducible with the $S_{\infty}$-universal orbit ER.

Concerning the finite-dimensional homeomorphism classification problems, Chang and Gao proved in \cite{ChangGaoDim} that both the homeomorphism ER and the ambient homeomorphism ER of compacta in $[0,1]^n$ are Borel reducible to the homeomorphism ER of continua in $[0,1]^{n+2}$ for every $n \in \N$. They also showed that the ambient homeomorphism ER of compacta in $[0,1]^n$ is Borel reducible to the ambient homeomorphism ER of compacta in $[0,1]^{n+1}$ for every $n \in \N$ and it is strictly more complex than the $S_{\infty}$-universal orbit ER when $n>1$.

The main results of our investigation follow:
\begin{itemize}
\item[{[A]}]  The homeomorphism ER of absolute retracts in $\R^2$ is Borel bireducible with the isomorphism ER of countable graphs (Theorem \ref{thmAR2}).
\item[{[B]}] The homeomorphism ER of absolute retracts in $\R^3$ is not classifiable by countable structures (Theorem \ref{thmAR3}).
\item[{[C]}] The homeomorphism ER of locally connected 1-dimensional continua in the plane is not classifiable by countable structures (Theorem \ref{thmLC1}).
\item[{[D]}] The homeomorphism ER of compacta in $\R^n$ is Borel reducible to the homeomorphism ER of $n$-dimensional continua in $\R^{n+1}$ (Theorem \ref{thmCompCont}). 
\item[{[E]}] The ambient homeomorphism ER of compacta in $[0,1]^n$ is Borel reducible to the homeomorphism ER of $n$-dimensional continua in $\R^{n+1}$ (Theorem \ref{AmbientHomeos}).
\end{itemize}

Claim [A] is a generalization of \cite[Theorem 6.7]{CamerloDarjiMarcone}.
Results [B] and [C] witness the sharpness of [A].
Claims [D] and [E] strengthen \cite[Theorem 1]{ChangGaoDim}.

It is worth noting that the Gelfand duality establishes a direct relationship between the complexity of compacta up to homeomorphism and the complexity of commutative unital $C^*$-algebras up to isomorphism. It is often the case that there are natural non-commutative versions of compacta (e.g. in the case of absolute retracts or dendrites \cite{ChigogidzeDranishnikov}). 
Hence, complexity results in topology can serve as conjectures for complexity results on $C^*$-algebras and vice versa (compare e.g. \cite{Sabok} and \cite{Zielinski}).
However, the most intriguing question concerns the complexity of compacta with bounded dimension.

\begin{question}
What is the exact complexity of the homeomorhism ER of $n$-dimensional compacta for $n\geq 1$?
\end{question}

\section{Preliminaries}

In this section we introduce the notation, terminology, definitions and basic facts which will be used throughout this paper.

By a natural number we mean a strictly positive integer. We denote the set of natural numbers by $\N$. In addition, we use the symbol ${\N}_0$ to denote the set of non-negative integers. Therefore, ${\N}_0 = \N \cup \{ 0 \}$.

Recall that a standard Borel space is a measurable space $(X,\mathcal{A})$ for which there exists a Polish topology $\tau$ on $X$ such that the family of all Borel subsets of $(X , \tau)$ is equal to $\mathcal{A}$. It is well-known that for every Polish space $X$ and every Borel set $B \subseteq X$ the measurable space $\big(B , \{ A \subseteq B \, ; \ A \textup{ is Borel in } X \} \big)$ is a standard Borel space. Therefore, every Borel subset of a Polish space can be naturally viewed as a standard Borel space.

In order to compare the complexities of equivalence relations on standard Borel spaces we use the notion of Borel reducibility.

\begin{definition}\label{DefinitionOfReduction}
Let $X$, $Y$ be sets and let $E$, $F$ be equivalence relations on $X$ and $Y$ respectively. A mapping $f \colon X \to Y$ is called a reduction from $E$ to $F$ if for every two points $x , x' \in X$ we have $x E x' \iff f(x) F f(x')$.
\end{definition}

\begin{definition}\label{DefinitionOfReducibility}
Let $X$, $Y$ be Polish spaces and $E$, $F$ equivalence relations on $X$, $Y$ respectively. We say that $E$ is continuously reducible to $F$, and write $E {\leq}_c F$, if there is a continuous reduction from $E$ to $F$. We say that $E$ is continuously bireducible with $F$ if $E {\leq}_c F$ and $F {\leq}_c E$.
\end{definition}

\begin{definition}\label{DefinitionOfBorelReducibility}
Let $X$, $Y$ be standard Borel spaces and $E$, $F$ equivalence relations on $X$, $Y$ respectively. We say that $E$ is Borel reducible to $F$, and write $E {\leq}_B F$, if there is a Borel measurable reduction from $E$ to $F$. We say that $E$ is Borel bireducible with $F$ if $E {\leq}_B F$ and $F {\leq}_B E$.
\end{definition}

An equivalence relation $E$ on a standard Borel space is said to be classifiable by countable structures if there is a countable relation language $\mathcal{L}$ such that $E$ is Borel reducible to the isomorphism equivalence relation of $\mathcal{L}$-structures whose underlying set is $\N$.

Given a class $\mathcal{C}$ of equivalence relations on standard Borel spaces and an element $E \in \mathcal{C}$, we say that $E$ is universal for $\mathcal{C}$ if $F {\leq}_B E$ for every $F \in \mathcal{C}$.

For a Polish group $G$, a standard Borel space $X$ and a Borel action $\varphi \colon G \times X \to X$ we can consider the orbit equivalence relation induced by $\varphi$, that is, the equivalence relation $E$ on $X$ given by
\[ x E x' \iff \exists \, g \in G : \varphi (g,x) = x' . \]

It is known that for every Polish group $G$ there exists an equivalence relation $E_G$ which is universal for the class of all orbit ERs induced by Borel actions of $G$ on standard Borel spaces (see \cite[Theorem 5.1.8]{Gao}). Some of the most studied ERs in the field of invariant descriptive set theory include $E_{S_{\infty}}$ and $E_{G_{\infty}}$, where $S_{\infty}$ stands for the symmetric group on $\N$ and $G_{\infty}$ stands for the universal Polish group. An ER on a standard Borel space is classifiable by countable structures if and only if it is Borel reducible to $E_{S_{\infty}}$ (see \cite[Theorem 2.39]{Hjorth}). The equivalence relation $E_{G_{\infty}}$ is universal for the class of all orbit ERs induced by Borel actions of Polish groups on standard Borel spaces (see \cite[Theorem 5.1.9]{Gao}).

The Hilbert cube is the space $[0,1]^{\N}$ endowed with the product topology. We denote the Hilbert cube by $Q$. Recall that every Polish space is homeomorphic to a $G_{\delta}$ subset of $Q$. For a Polish space $X$, we denote by $\mathcal{K} (X)$ the space of all compact subsets of $X$ equipped with the Vietoris topology. It is well-known that $\mathcal{K} (X)$ is a Polish space. We consider various subspaces of $\mathcal{K} (X)$: We denote by $\mathsf{C} (X)$ the space of all continua in $X$ and by $\mathsf{LC} (X)$ the space of all locally connected continua in $X$. For every $n \in {\N}_0$ we denote by ${\mathsf{C}}_n (X)$ the subspace of $\mathsf{C} (X)$ consisting of those members of $\mathsf{C} (X)$ which are $n$-dimensional. In similar fashion, we define the space ${\mathsf{LC}}_n (X)$, $n \in {\N}_0$. It is an easy exercise that $\mathsf{C} (X)$ is a closed (and thus Borel) subset of $\mathcal{K} (X)$. The set $\mathsf{LC} (X)$ is Borel in $\mathcal{K} (X)$ as well (see \cite{GvM} for reference). Moreover, it is not difficult to show that the set $\{ K \in \mathcal{K} (Q) ; \, \mathrm{dim} K \leq n \}$ is $G_{\delta}$ in $\mathcal{K} (Q)$ for every $n \in {\N}_0$. It follows that for every Polish space $X$ and every $n \in {\N}_0$ the set $\{ K \in \mathcal{K} (X) ; \, \mathrm{dim} K = n \}$ is Borel in $\mathcal{K} (X)$ and (consequently) so are the sets ${\mathsf{C}}_n (X)$ and ${\mathsf{LC}}_n (X)$.

An absolute retract is a topological space which is homeomorphic to a retract of the Hilbert cube. For a Polish space $X$, we denote by $\mathsf{AR} (X)$ the set of all absolute retracts contained in $X$. By \cite[Theorem 2.2]{DobrowolskiRubin}, the set $\mathsf{AR} (Q)$ is $G_{\delta\sigma\delta}$ in $\mathcal{K} (Q)$. Hence, the set $\mathsf{AR} (X)$ is Borel in $\mathcal{K} (X)$ for every Polish space $X$.

For a Polish space $X$ and a Borel set $B \subseteq \mathcal{K} (X)$ the equivalence relation
\[ \big\lbrace (K,L) \in B \times B \, ; \ K \textup{ is homeomorphic to } L \big\rbrace \]
is called the homeomorphism equivalence relation on $B$. In addition, we consider the equivalence relation
\[ \big\lbrace (K,L) \in B \times B \, ; \ \textup{there is a self-homeomorphism of } X \textup{ mapping } K \textup{ onto } L \big\rbrace \]
and call it the ambient homeomorphism equivalence relation on $B$.

For every $n \in \N$ we denote by $H_n$ and $C_n$ the homeomorphism ERs on $\mathcal{K} \big( [0,1]^n \big)$ and $\mathsf{C} \big( [0,1]^n \big)$ respectively. In addition, $R_n$ stands for the ambient homeomorphism ER on $\mathcal{K} \big( [0,1]^n \big)$.

Since the Hilbert cube contains a homeomorphic copy of every metrizable compact space, the homeomorphism ER on $\mathcal{K} (Q)$ is often interpreted as the homeomorphism ER of all metrizable compacta. The homeomorphism ERs on $\mathsf{C} (Q)$, $\mathsf{LC} (Q)$, $\mathsf{AR} (Q)$, etc., can be interpreted in a similar fashion.

Recall that a set $J \subseteq \R^2$ is called a Jordan curve if it is homeomorphic to a circle. By the Jordan curve theorem, for every Jordan curve $J$ the set $\R^2 \setminus J$ consists of exactly two connected components – one bounded and one unbounded. We denote the bounded one by $\ins (J)$ and the unbounded one by $\out (J)$. The Jordan curve theorem also states that the boundary of both $\ins (J)$ and $\out (J)$ is equal to $J$. The following lemma immediately follows.

\begin{lemma}\label{JordUnique}
Let $J_1 , J_2 \subseteq \R^2$ be Jordan curves such that $\ins (J_1) = \ins (J_2)$. Then $J_1 = J_2$.
\end{lemma}

The Jordan-Schoenflies theorem states that for every Jordan curve $J$ there is a self-homeomorphism of $\R^2$ which maps $J$ onto the unit circle (and, consequently, $\ins (J)$ is mapped onto the open unit disk). The following lemma is corollary of this theorem.

\begin{lemma}\label{extension}
Let $J_1 , J_2 \subseteq \R^2$ be Jordan curves and assume that $h\colon J_1 \to J_2$ is an onto homeomorphism. Then there exists an onto homeomorphism $\widetilde{h} \colon J_1 \cup \ins (J_1) \to J_2 \cup \ins (J_2)$ such that $\widetilde{h} (x) = h (x)$ for every $x \in J_1$.
\end{lemma}

\begin{proof}
Let $U:=\lbrace x \in \R^2 \, ; \ \|x\| < 1 \rbrace$. By the Jordan-Schoenflies theorem, there are onto homeomorphisms $h_1 , h_2 \colon \R^2 \to \R^2$ satisfying
\[h_1 (J_1) = h_2 (J_2) = \partial U \, , \ \ h_1 \big( \ins (J_1) \big) = h_2 \big( \ins (J_2) \big) = U.\]
Denote $y_1:=h_1^{-1} \big( (0,0) \big)$, $y_2:=h_2^{-1} \big( (0,0) \big)$ and let us define a mapping $\widetilde{h} \colon J_1 \cup \ins (J_1) \to J_2 \cup \ins (J_2)$ by
\begin{equation*}
    \widetilde{h} (x) = \begin{cases}
    h_2^{-1} \bigg( \| h_1 (x) \| \cdot \big( h_2 \circ h \circ h_1^{-1} \big) \Big ( \frac{h_1 (x)}{\| h_1 (x) \|} \Big ) \bigg ) & \textup{if } x \neq y_1 \\
    \\
    y_2 & \textup{if } x=y_1 .
\end{cases} \ \ 
\end{equation*}
It is easy to verify that $\widetilde{h}$ is the desired homeomorphism.
\end{proof}

Let us close this chapter by the following lemma, whose proof is an easy exercise.
\begin{lemma}\label{ClosednessOfUnion}
Let $M$ be a metric space, $F \subseteq M$ a closed set and $\mathcal{A} \subseteq \mathcal{P} (M)$ a family such that $\overline{A} \cap F \neq \emptyset$ and $\overline{A} \setminus A \subseteq F$ for every $A \in \mathcal{A}$. If the set $\big\lbrace A \in \mathcal{A} \, ; \, \mathrm{diam} (A) \geq \varepsilon \big\rbrace$ is finite for every $\varepsilon > 0$, then the set $F \cup \bigcup \mathcal{A}$ is closed in $M$.
\end{lemma}

\section{Classification results}

In this section we prove that the homeomorphism ER of absolute retracts in the plane is Borel bireducible with the isomorphism ER of countable graphs. We also prove that both the homeomorphism and ambient homeomorphism ERs of compacta in $[0,1]^n$ are Borel reducible to the homeomorphism ER of $n$-dimensional continua in $[0,1]^{n+1}$. Let us start by presenting some results concerning locally connected continua and absolute retracts. The following characterization of planar absolute retracts can be found in \cite[p. 132]{Borsuk}.

\begin{lemma}\label{PlanarAR}
Let $X \subseteq \R^2$. Then $X$ is an absolute retract if and only if $X$ is a locally connected continuum and ${\R}^2 \setminus X$ is connected.
\end{lemma}

Recall that a topological space is said to be rim-finite if it has an open base consisting of sets with finite boundaries. Combining Lemma \ref{PlanarAR} with \cite[p. 512, Theorem 4]{Kuratowski2}, we obtain the following result.

\begin{lemma}\label{rim-finite}
Let $X \subseteq \R^2$ be an absolute retract. Then $\partial X$ is a rim-finite continuum.
\end{lemma}

The next lemma can be found in \cite[p. 33]{Whyburn} labeled as the Torhorst theorem.

\begin{lemma}\label{Torhorst}
Let $X \subseteq \R^2$ be a locally connected continuum and $G \subseteq \R^2 \setminus X$ a connected component of $\R^2 \setminus X$. Then $\partial G$ is a locally connected continuum.
\end{lemma}

Combining Lemmata \ref{PlanarAR} and \ref{Torhorst} we obtain the following:

\begin{lemma}\label{BdOfARisLC}
Let $X \subseteq \R^2$ be an absolute retract. Then $\partial X$ is a locally connected continuum.
\end{lemma}

The following lemma easily follows from \cite[p. 513, Theorem 5]{Kuratowski2}.

\begin{lemma}\label{JordSeparace}
Let $C \subseteq \R^2$ be a locally connected continuum such that $\R^2 \setminus C$ is disconnected. Let $A,B \subseteq \R^2 \setminus C$ be continua such that the connected component of $\R^2 \setminus C$ containing $A$ does not contain $B$. Then there exists a Jordan curve $J \subseteq C$ such that one of the sets $A,B$ is contained in $\ins (J)$ and the other one is contained in $\out (J)$.
\end{lemma}

\begin{lemma}\label{Komponenta}
Let $T$ be a topological space and $S \subseteq T$. Then every connected component of $\Int (S)$ is also a connected component of $T \setminus \partial S$.
\end{lemma}
\begin{proof}
Suppose that $C$ is a connected component of $\Int (S)$. Since $\Int (S)$ is a subset of $T \setminus \partial S$, so is $C$. Let $C_0$ be the connected component of $T \setminus \partial S$ which contains $C$. To complete the proof it suffices to show that $C_0 \subseteq \Int (S)$. Consider the sets $C_1 := C_0 \setminus \overline{S}$, $C_2 := C_0 \cap \Int (S)$. They are disjoint, relatively open in $C_0$ and their union is equal to $C_0$. Since $C_0$ is connected and $C_2 \supseteq C_0 \cap C = C \neq \emptyset$, it follows that $C_1 = \emptyset$. Hence, $C_2 = C_0$, which shows that $C_0 \subseteq \Int (S)$.
\end{proof}

\begin{lemma}\label{insAndInt}
Let $X \subseteq \R^2$ be an absolute retract. Then
\[\big\lbrace G \subseteq \Int (X) \, ; \ G \textup{ is a connected component of } \Int (X) \big\rbrace\]
\[=\big\lbrace \ins (J) \, ; \ J \subseteq \partial X \textup{ is a Jordan curve} \big\rbrace .\]
\end{lemma}
\begin{proof}
Let us start by showing that $\ins (J)$ is a subset of $\Int (X)$ for every Jordan curve $J$ contained in $\partial X$. Assume for contradiction that $J \subseteq \partial X$ is a Jordan curve such that $\ins (J) \nsubseteq \Int (X)$. Then, as $\ins (J)$ is open, we have $\ins (J) \nsubseteq X$. Hence, $(\R^2 \setminus X) \cap \ins (J) \neq \emptyset$. Moreover, since $X$ is compact (and hence bounded in $\R^2$), we have $(\R^2 \setminus X) \cap \out (J) \neq \emptyset$. In addition, the sets $\ins (J)$, $\out (J)$ are connected and (by Lemma~\ref{PlanarAR}) so is $\R^2 \setminus X$. Therefore, the set $\ins (J) \cup (\R^2 \setminus X) \cup \out (J)$ is connected as well. However, it is clear that this set is equal to $\R^2 \setminus J$. Thus, $\R^2 \setminus J$ is connected which contradicts the Jordan curve theorem.

Now let us prove that $\ins (J)$ is a connected component of $\Int (X)$ for every Jordan curve $J \subseteq \partial X$. Let $J$ be given. We know that $\ins (J)$ is a connected subset of $\Int (X)$. Let $U$ be the connected component of $\Int (X)$ containing $\ins (J)$, we need to show that $U \subseteq \ins (J)$. Assume that $U \setminus \ins (J) \neq \emptyset$. As $J \subseteq \partial X$ and $U \subseteq \Int (X)$, we have $U \cap J = \emptyset$, therefore $U \cap \out (J) \neq \emptyset$. It follows that the set $U \cup \out (J)$ is connected. Thus, since it is clear that $U \cup \out (J) = \R^2 \setminus J$, we see that $\R^2 \setminus J$ is connected, which contradicts the Jordan curve theorem.

It remains to prove that for every connected component $G$ of $\Int (X)$ there is a Jordan curve $J \subseteq \partial X$ such that $G = \ins (J)$. Let $G$ be given. By Lemma~\ref{BdOfARisLC}, the set $\partial X$ is a locally connected continuum. Hence, since $G$ is (by Lemma~\ref{Komponenta}) a connected component of $\R^2 \setminus \partial X$, it follows from Lemma~\ref{Torhorst} that $\partial G$ is a locally connected continuum. As a connected component of an open subset of Euclidean space, the set $G$ is open. Thus, since $G$ clearly is a connected component of itself, we conclude using Lemma~\ref{Komponenta} that $G$ is a connected component of $\R^2 \setminus \partial G$. Choose arbitrary points $a \in G$, $b \in \R^2 \setminus X$. Clearly, $b \notin \overline{G}$. Therefore, $b$ lies in $\R^2 \setminus \partial G$ and the connected component of $\R^2 \setminus \partial G$ containing $a$ does not contain $b$. Let $A := \{ a \}$, $B := \{ b \}$. Then $A$, $B$ are continua such that $A,B \subseteq \R^2 \setminus \partial G$ and the connected component of $\R^2 \setminus \partial G$ containing $A$ does not contain $B$. Thus, by Lemma~\ref{JordSeparace}, there is a Jordan curve $J \subseteq \partial G$ such that one of the sets $A$, $B$ is contained in $\ins (J)$ and the other one is contained in $\out (J)$. Clearly, $\partial G \subseteq \partial X$, hence $J \subseteq \partial X$. Therefore, by the previous paragraph, the set $\ins (J)$ is a connected component of $\Int (X)$. In particular, we have $\ins (J) \subseteq X$, hence $B \nsubseteq \ins (J)$. Thus, $A \subseteq \ins (J)$, which shows that $G \cap \ins (J) \neq \emptyset$. Hence, since both of the sets $G$, $\ins (J)$ are connected components of $\Int (X)$, we conclude that $G = \ins (J)$.
\end{proof}

The following lemma is an easy consequence of \cite[p. 515, Theorem 10]{Kuratowski2}. Before we state it, recall that an open subset of $\R^n$ has at most countably many connected components (as the components form a disjoint family of open sets and $\R^n$ is separable).

\begin{lemma}\label{DiamGoesToZero1}
Let $Z \subseteq \R^2$ be a locally connected compact set. Then the family
\[ \{ U \subseteq \R^2 \setminus Z \, ; \ \mathrm{diam} (U) \geq \varepsilon \, , \ U \textup{ is a connected component of } \R^2 \setminus Z \} \]
is finite for every $\varepsilon > 0$.
\end{lemma}

By Lemmata \ref{JordUnique}, \ref{BdOfARisLC}, \ref{Komponenta}, \ref{insAndInt} and \ref{DiamGoesToZero1}, we receive the following:

\begin{lemma}\label{DiamGoesToZero2}
Let $X \subseteq \R^2$ be an absolute retract. Then the set
\[ \big\lbrace J \subseteq \partial X \, ; \ J \textup{ is a Jordan curve and } \mathrm{diam} \big( \ins (J) \big) \geq \varepsilon \big\rbrace \]
is finite for every $\varepsilon > 0$.
\end{lemma}

\begin{lemma}\label{BorelMeas}
Let $X$ be a Polish space. Then the mapping from $\mathcal{K} (X)$ to $\mathcal{K} (X)$ given by $K \mapsto \partial K$ is Borel measurable.
\end{lemma}

\begin{proof}
Denote by $\Phi$ the mapping in question and let $\varrho$ be a metric on $X$ compatible with the topology of $X$. For every open set $G \subseteq X$ denote
\[ A_G := \big\lbrace L \in \mathcal{K} (X) \, ; \, L \subseteq G \big\rbrace \ , \ \ \ \ B_G := \big\lbrace L \in \mathcal{K} (X) \, ; \, L \cap G \neq \emptyset \big\rbrace . \]
By the definition of the Vietoris topology, the family
\[ \mathcal{S} := \{ A_G \, ; \, G \subseteq X \textup{ is open} \} \cup \{ B_G \, ; \, G \subseteq X \textup{ is open} \} \]
is a subbase of $\mathcal{K} (X)$.
For every $K \in \mathcal{K} (X)$ and $n \in \N$ denote
$${\partial}_n K := \big\lbrace x \in K \, ; \ \exists \, y \in X \setminus K : \varrho (x,y) < 2^{-n} \big\rbrace .$$
Clearly, we have
\begin{equation}\label{eq1}
    \forall K \in \mathcal{K} (X) \ \forall n \in \N : {\partial}_n K \supseteq \overline{{\partial}_{n+1} K}.
\end{equation}
Moreover, it is easy to see that
\begin{equation}\label{eq2}
    \forall K \in \mathcal{K} (X) : \bigcap_{n=1}^{\infty} {\partial}_n K = \bigcap_{n=1}^{\infty} \overline{{\partial}_n K} = \partial K .
\end{equation}
For every $n \in \N$ define a mapping ${\Phi}_n \colon \mathcal{K} (X) \to \mathcal{K} (X)$ by ${\Phi}_n (K) = \overline{{\partial}_n K}$. We will show that the sequence $({\Phi}_n)_{n=1}^{\infty}$ converges pointwise to $\Phi$. Let $K \in \mathcal{K} (X)$ be given. In order to prove that $\displaystyle \lim_{n \to \infty} {\Phi}_n (K) = \Phi (K)$ in $\mathcal{K} (X)$, it suffices to show that for every $U \in \mathcal{S}$ with $\partial K \in U$ there is $n_0 \in \N$ such that for every $n \in \N$ with $n \geq n_0$ we have $\overline{{\partial}_n K} \in U$. Let $U \in \mathcal{S}$ with $\partial K \in U$ be arbitrary. By the definition of $\mathcal{S}$, there is an open set $G \subseteq X$ such that $U = A_G$ or $U = B_G$. First assume that $U = A_G$. Since $\partial K \in U$, we have $\partial K \subseteq G$. Hence, thanks to \eqref{eq1} and \eqref{eq2}, standard compactness arguments show that there exists $n_0 \in \N$ such that $\overline{{\partial}_{n_0} K} \subseteq G$. Thus, by \eqref{eq1}, we have $\overline{{\partial}_n K} \subseteq G$ for every $n \in \N$ with $n \geq n_0$. In other words, $\overline{{\partial}_n K} \in A_G = U$ for every $n \in \N$ with $n \geq n_0$. Now let us consider the case that $U = B_G$. Since $\partial K \in U$, we have $\partial K \cap G \neq \emptyset$. Using this and \eqref{eq2} we immediately conclude that $\overline{{\partial}_n K} \cap G \neq \emptyset$ for every $n \in \N$. In other words, $\overline{{\partial}_n K} \in B_G = U$ for every $n \in \N$. Thus, we can take $n_0 = 1$ and we are done.

Having shown that $\Phi$ is the pointwise limit of the sequence $({\Phi}_n)_{n=1}^{\infty}$, it only remains to prove that each of the mappings ${\Phi}_n$, $n \in \N$, is Borel measurable. Let $n \in \N$ be given. Since $\mathcal{K} (X)$ is second-countable (as it is a Polish space), it suffices to show that ${\Phi}_n^{-1} (U)$ is Borel in $\mathcal{K} (X)$ for every $U \in \mathcal{S}$. To that end we will use the following two claims:

\begin{claim}\label{ClaimOpen}
For every open set $G \subseteq X$ the set $\{ K \in \mathcal{K} (X) \, ; \ {\partial}_n K \cap G \neq \emptyset \}$\\
is open in $\mathcal{K} (X)$.
\end{claim}

\begin{claimproof}
Let $G \subseteq X$ be an open set. Choose an arbitrary set $K \in \mathcal{K} (X)$ such that ${\partial}_n K \cap G \neq \emptyset$. We would like to find a neighbourhood $V$ of $K$ in $\mathcal{K} (X)$ such that every $L \in V$ satisfies ${\partial}_n L \cap G \neq \emptyset$. Fix a point $x \in {\partial}_n K \cap G$. By the definition of ${\partial}_n K$, there exists $y \in X \setminus K$ such that $\varrho (x,y) < 2^{-n}$. Let ${\varepsilon}_1 := 2^{-n}-\varrho (x,y)$. Since $G$ is open and $x \in G$, there is ${\varepsilon}_2 > 0$ such that every $z \in X$ with $\varrho(x,z) < {\varepsilon}_2$ belongs to $G$. Let $\varepsilon := \mathrm{min} \{ {\varepsilon}_1, {\varepsilon}_2 \}$ and denote $B := \{ z \in X ; \, \varrho (x,z) < \varepsilon \}$. Clearly, $B$ is an open subset of $G$. Define
\[ V := \big\lbrace L \in \mathcal{K} (X) \, ; \ L \subseteq X \setminus \{ y \} \, , \ L \cap B \neq \emptyset \big\rbrace . \]
By the definition of the Vietoris topology, $V$ is open in $\mathcal{K} (X)$. Moreover, as $y \notin K$ and $x \in {\partial}_n K \cap B \subseteq K \cap B$, the set $V$ is a neighbourhood of $K$. Given $L \in V$, let us show that ${\partial}_n L \cap G \neq \emptyset$. By the definition of $V$, there exists a point $z \in L \cap B$. We have $\varrho (x,z) < \varepsilon \leq {\varepsilon}_1 = 2^{-n}-\varrho (x,y)$, and thus $\varrho (z,y) < 2^{-n}$. Hence, as $z \in L$ and $y \in X \setminus L$, we find that $z \in {\partial}_n L$. Therefore, since $z \in B \subseteq G$, it follows that ${\partial}_n L \cap G \neq \emptyset$. \claimend
\end{claimproof}

\begin{claim}\label{ClaimClosed}
For every closed set $F \subseteq X$ the set $\{ K \in \mathcal{K} (X) \, ; \ \overline{{\partial}_n K} \cap F \neq \emptyset \}$\\
is $G_{\delta}$ in $\mathcal{K} (X)$.
\end{claim}

\begin{claimproof}
Let $F \subseteq X$ be a closed set. For every $k \in \N$ let
\[ G_k := \big\lbrace x \in X ; \ \exists \, y \in F : \varrho (x,y) < 2^{-k} \big\rbrace . \]
It is clear that each of the sets $G_k$, $k \in \N$, is an open superset of $F$. Let us show that
\begin{equation}\label{eq3}
   \{ K \in \mathcal{K} (X) \, ; \ \overline{{\partial}_n K} \cap F \neq \emptyset \} = \bigcap_{k=1}^{\infty} \{ K \in \mathcal{K} (X) \, ; \ \overline{{\partial}_n K} \cap G_k \neq \emptyset \} .
\end{equation}
It is clear that every $K \in \mathcal{K} (X)$ with $\overline{{\partial}_n K} \cap F \neq \emptyset$ satisfies $\overline{{\partial}_n K} \cap G_k \neq \emptyset$ for every $k \in \N$. Conversely, let $K \in \mathcal{K} (X)$ be a set such that $\overline{{\partial}_n K} \cap F = \emptyset$. As a closed subset of the compact set $K$, the set $\overline{{\partial}_n K}$ is compact. Hence, there is $\varepsilon > 0$ such that for every $x \in \overline{{\partial}_n K}$ and every $y \in X$ with $\varrho (x,y) < \varepsilon$ we have $y \notin F$. It follows that $\overline{{\partial}_n K} \cap G_k = \emptyset$ for every $k \in \N$ with $2^{-k} \leq \varepsilon$.

Having shown \eqref{eq3}, it suffices to prove that $\{ K \in \mathcal{K} (X) \, ; \ \overline{{\partial}_n K} \cap G_k \neq \emptyset \}$ is an open set for every $k \in \N$. However, since each of the sets $G_k$, $k \in \N$, is open, we clearly have
\[ \{ K \in \mathcal{K} (X) \, ; \ \overline{{\partial}_n K} \cap G_k \neq \emptyset \} = \{ K \in \mathcal{K} (X) \, ; \ {\partial}_n K \cap G_k \neq \emptyset \} \]
for every $k \in \N$. Thus, thanks to Claim \ref{ClaimOpen}, we are done.\claimend
\end{claimproof}
Now we are ready to finish the proof. Let $U \in \mathcal{S}$ be given. By the definition of $\mathcal{S}$, there is an open set $G \subseteq X$ such that $U = A_G$ or $U = B_G$. First assume that $U = A_G$. Then ${\Phi}_n^{-1} (U) = \{ K \in \mathcal{K} (X) \, ; \ \overline{{\partial}_n K} \subseteq G \}$ and it is clear that
\[ \mathcal{K} (X) \setminus \{ K \in \mathcal{K} (X) \, ; \ \overline{{\partial}_n K} \subseteq G \} = \{ K \in \mathcal{K} (X) \, ; \ \overline{{\partial}_n K} \cap (X \setminus G) \neq \emptyset \} . \]
Therefore, by Claim \ref{ClaimClosed} and the closedness of $X \setminus G$, the set $\mathcal{K} (X) \setminus {\Phi}_n^{-1} (U)$ is $G_{\delta}$. In particular, ${\Phi}_n^{-1} (U)$ is Borel. Finally, assume that $U = B_G$. Then
\[ {\Phi}_n^{-1} (U) = \{ K \in \mathcal{K} (X) \, ; \ \overline{{\partial}_n K} \cap G \neq \emptyset \} = \{ K \in \mathcal{K} (X) \, ; \ {\partial}_n K \cap G \neq \emptyset \} , \]
hence, by Claim \ref{ClaimOpen}, the set ${\Phi}_n^{-1} (U)$ is open. In particular, it is Borel.
\end{proof}

\begin{theorem}\label{thmAR2}
The homeomorphism equivalence relation on $\mathsf{AR} ( \R^2 )$ is Borel bireducible with the isomorphism equivalence relation of countable graphs.
\end{theorem}

\begin{proof}
By \cite[Theorem 6.7]{CamerloDarjiMarcone}, the homeomorphism ER of dendrites is Borel bireducible with the isomorphism ER of countable graphs.  Moreover, by a careful reading of \cite[Lemma 6.6]{CamerloDarjiMarcone} and its proof we conclude that the homeomorphism ER of dendrites in $[0,1]^2$ is Borel bireducible with the isomorphism ER of countable graphs (in fact, by \cite[paragraph 10.37]{Nadler}, every dendrite can be embedded into the plane). 
By \cite[p. 344, Theorem 16]{Kuratowski2}, every dendrite is an absolute retract.
Therefore, the isomorphism ER of countable graphs is Borel reducible to the homeomorphism ER on $\mathsf{AR} ( \R^2 )$.

To prove that the homeomorphism ER on $\mathsf{AR} ( \R^2 )$ is Borel reducible to the isomorphism ER of countable graphs, it suffices to show – thanks to \cite{KrupskiVejnar} – that the homeomorphism ER on $\mathsf{AR} ( \R^2 )$ is Borel reducible to the homeomorphism ER of rim-finite continua. Hence, it suffices to find a Borel measurable mapping $\Phi \colon \mathsf{AR} ( \R^2 ) \to \mathcal{K} (\R^2)$ such that $\Phi (X)$ is a rim-finite continuum for every $X \in \mathsf{AR} ( \R^2 )$ and $\Phi$ is a reduction from the homeomorphism ER on $\mathsf{AR} ( \R^2 )$ to the homeomorphism ER on $\mathcal{K} (\R^2)$. Define this desired mapping $\Phi$ by $\Phi (X) = \partial X$, $X \in \mathsf{AR} ( \R^2 )$. By Lemma \ref{BorelMeas}, we have that $\Phi$ is Borel measurable. Moreover, by Lemma \ref{rim-finite}, the set $\Phi (X)$ is a rim-finite continuum for every $X \in \mathsf{AR} ( \R^2 )$. Thus, it only remains to show that $\Phi$ is a reduction.

Suppose that $X_1 , X_2 \in \mathsf{AR} ( \R^2 )$ are homeomorphic and let $f \colon X_1 \to X_2$ be an onto homeomorphism. In order to show that $\Phi (X_1)$ is homeomorphic to $\Phi (X_2)$ it suffices to prove that $f (\partial X_1) = \partial X_2$. By the domain invariance theorem, the sets $f \big( \Int (X_1) \big)$ and $f^{-1} \big( \Int (X_2) \big)$ are open in $\R^2$. It follows that $f \big( \Int (X_1) \big) \subseteq \Int (X_2)$ and $f^{-1} \big( \Int (X_2) \big) \subseteq \Int (X_1)$. Therefore, $f \big( \Int (X_1) \big) = \Int (X_2)$, which implies that $f (\partial X_1) = \partial X_2$.

Conversely, assume we are given sets $X_1 , X_2 \in \mathsf{AR} ( \R^2 )$ such that $\Phi (X_1)$ is homeomorphic to $\Phi (X_2)$. Let $h \colon \partial X_1 \to \partial X_2$ be an onto homeomorphism. For every $i \in \{ 1 , 2 \}$ denote by ${\Gamma}_i$ the set of all Jordan curves contained in $\partial X_i$. Since $h$ is an onto homeomorphism, it is clear that the mapping $J \mapsto h(J)$, $J \in {\Gamma}_1$, is a bijection between ${\Gamma}_1$ and ${\Gamma}_2$. Moreover, trivially, for every $J \in {\Gamma}_1$ the mapping $h\restriction_J$ is a homeomorphism between the Jordan curves $J$ and $h(J)$. Hence, by Lemma \ref{extension}, for every $J \in {\Gamma}_1$ there is an onto homeomorphism $g_J \colon J \cup \ins (J) \to h(J) \cup \ins \big( h(J) \big)$ such that $g_J \restriction_J = h\restriction_J$. By Lemma \ref{insAndInt}, the set $\ins (J)$ is a subset of $\Int (X_i)$ for every $J \in {\Gamma}_i$ and every $i \in \{ 1 , 2 \}$. Moreover, by Lemmata \ref{JordUnique} and \ref{insAndInt}, for every $i \in \{ 1 , 2 \}$ and every $x \in \Int (X_i)$ there is a unique Jordan curve $J \in {\Gamma}_i$ such that $x \in \ins (J)$. It follows from this and from the properties of the mappings $h$ and $g_J$, $J \in {\Gamma}_1$, that the mapping $f \colon X_1 \to X_2$ given by
\begin{equation*}
    f(x) = \begin{cases}
    h(x) & \textup{if } x \in \partial X_1
    \\
    g_J (x) & \textup{if } x \in \ins (J) \, , \, J \in {\Gamma}_1
\end{cases} \ \ 
\end{equation*}
is a well-defined bijection. Let us show that $f$ is a homeomorphism. Since $X_1$ is compact and $f$ is injective, it suffices to show that $f$ is continuous. Let $x_0 \in X_1$ be given, we would like to prove that $f$ is continuous at $x_0$.

First, assume that $x_0 \in \Int (X_1)$. Then there is a unique Jordan curve $J \in {\Gamma}_1$ with $x_0 \in \ins (J)$. By the definition of $f$, we have $f \restriction_{\ins (J)} = g_J \restriction_{\ins (J)}$. Hence, as $\ins (J)$ is open and $g_J$ is continuous, the mapping $f$ is continuous at every point $x \in \ins (J)$. In particular, it is continuous at $x_0$.

Now, suppose that $x_0 \in \partial X_1$ and let $\varepsilon > 0$ be given. By the continuity of $h$, there is ${\delta}_1 > 0$ such that $\norm{h(x)-h(x_0)} < \frac{1}{2} \varepsilon$ for every $x \in \partial X_1$ with $\norm{x-x_0} < 2 {\delta}_1$. Consider the following sets:
\begin{align*}
    M_1 &:= \Big\lbrace J \in {\Gamma}_1 \, ; \ \mathrm{diam} \big( \ins (J) \big) \geq {\delta}_1 \Big\rbrace ,\\
    M_2 &:= \Big\lbrace J \in {\Gamma}_2 \, ; \ \mathrm{diam} \big( \ins (J) \big) \geq \frac{\varepsilon}{2} \Big\rbrace ,\\
    M &:= M_1 \cup \big\lbrace J \in {\Gamma}_1 \, ; \ h(J) \in M_2 \big\rbrace ,\\
    \widetilde{M} &:= \lbrace J \in M \, ; \ x_0 \in J \rbrace .
\end{align*}
By Lemma \ref{DiamGoesToZero2}, the sets $M_1$, $M_2$ are finite. Moreover, since the mapping $J \mapsto h(J)$, $J \in {\Gamma}_1$, is a bijection between ${\Gamma}_1$ and ${\Gamma}_2$, the set $M$ is finite as well (and so are the sets $\widetilde{M}$ and $M \setminus \widetilde{M}$). For every Jordan curve $J \in \widetilde{M}$ there exists (since $g_J$ is continuous) a number $\delta (J)>0$ such that every point $x \in J \cup \ins (J)$ with $\norm{x-x_0} < \delta (J)$ satisfies $\norm{g_J (x) - g_J (x_0)} < \varepsilon$. Define ${\delta}_2 := \mathrm{min} \big( \{ 1 \} \cup \{ \delta (J) \, ; \, J \in \widetilde{M} \} \big)$. By the definition of the set $\widetilde{M}$, it is obvious that $x_0 \notin J$ for every $J \in M \setminus \widetilde{M}$. In addition, as $x_0$ lies in $\partial X_1$ and $\ins (J) \subseteq \Int (X_1)$ for every $J \in {\Gamma}_1$, it follows that $x_0 \in \out (J)$ for every $J \in M \setminus \widetilde{M}$. As the set $M \setminus \widetilde{M}$ is finite and the set $\out (J)$ is open for every $J \in M \setminus \widetilde{M}$, there exists ${\delta}_3 > 0$ such that every $x \in \R^2$ with $\norm{x-x_0} < {\delta}_3$ lies in $\out (J)$ for every $J \in M \setminus \widetilde{M}$. Let $\delta := \mathrm{min} \{ {\delta}_1 , {\delta}_2 , {\delta}_3 \}$. We claim that $\norm{f(x)-f(x_0)} < \varepsilon$ for every $x \in X_1$ satisfying $\norm{x-x_0} < \delta$. Let $x \in X_1$ with $\norm{x-x_0} < \delta$ be given. If $x \in \partial X_1$, we have
\[ \norm{f(x)-f(x_0)} = \norm{h(x)-h(x_0)} < \tfrac{1}{2} \varepsilon < \varepsilon , \]
as $\norm{x-x_0} < \delta \leq {\delta}_1 < 2 {\delta}_1$. Therefore, we can assume that $x \in \Int (X_1)$, otherwise we are done. Let $J \in {\Gamma}_1$ be the (unique) Jordan curve satisfying $x \in \ins (J)$. Since $\norm{x-x_0} < \delta \leq {\delta}_3$, it is clear that $J \notin M \setminus \widetilde{M}$. Hence, $J \in \widetilde{M} \cup ({\Gamma}_1 \setminus M)$. First, assume that $J \in \widetilde{M}$. By the definition of the mapping $f$ and by the fact that $g_J \restriction_J = h\restriction_J$, we have $f(x) = g_J (x)$ and $f(x_0) = h(x_0) = g_J (x_0)$. Therefore, as $\norm{x-x_0} < \delta$ and $\delta \leq {\delta}_2 \leq \delta (J)$, we conclude that $\norm{f(x)-f(x_0)} = \norm{g_J (x) - g_J (x_0)} < \varepsilon$. Finally, consider the case that $J \in {\Gamma}_1 \setminus M$. By the definitions of the sets $M$, $M_1$ and $M_2$, we have $\mathrm{diam} \big( \ins (J) \big) < {\delta}_1$ and $\mathrm{diam} \big( \ins (h(J)) \big) < \frac{1}{2} \varepsilon$. Choose an arbitrary point $y \in J$. Then we have
\[ \norm{y-x} \leq \mathrm{diam} \big( J \cup \ins (J) \big) = \mathrm{diam} \big( \overline{\ins (J)} \big) = \mathrm{diam} \big( \ins (J) \big) < {\delta}_1 . \]
It follows that $\norm{y-x_0} \leq \norm{y-x} + \norm{x-x_0} < {\delta}_1 + \delta \leq 2 {\delta}_1$, therefore $\norm{h(y)-h(x_0)} < \frac{1}{2} \varepsilon$. Consequently, as $f(y)=h(y)$ and $f(x_0)=h(x_0)$, we obtain $\norm{f(y)-f(x_0)} < \frac{1}{2} \varepsilon$. Moreover, we also have
\begin{align*}
    \norm{f(x)-f(y)} &= \norm{g_J (x) - h(y)} = \norm{g_J (x) - g_J (y)}\\
    &\leq \mathrm{diam} \Big( g_J \big( J \cup \ins (J) \big)  \Big) = \mathrm{diam} \Big( h(J) \cup \ins \big( h(J) \big) \Big)\\
    &\leq \mathrm{diam} \Big( \overline{\ins \big( h(J) \big)} \Big) = \mathrm{diam} \Big( \ins \big( h(J) \big) \Big) < \tfrac{1}{2} \varepsilon .
\end{align*}
Thus, $\norm{f(x)-f(x_0)} \leq \norm{f(x)-f(y)} + \norm{f(y)-f(x_0)} < \frac{1}{2} \varepsilon + \frac{1}{2} \varepsilon = \varepsilon$.
\end{proof}

We devote the rest of this chapter to proving that both the homeomorphism and ambient homeomorphism ERs on $\mathcal{K} \big( [0,1]^n \big)$ are Borel reducible to the homeomorphism ER on ${\mathsf{C}}_n \big( [0,1]^{n+1} \big)$ for every $n \in \N$. In particular, using the notation introduced earlier, we show that $H_n \, {\leq}_B \, C_{n+1}$ and $R_n \, {\leq}_B \, C_{n+1}$ for every $n \in \N$. This result together with \cite[Theorem 3.9]{ChangGaoDim} and \cite[Theorem 4.2]{ChangGaoDim} give rise to the following diagram, where each arrow denotes Borel reducibility.

\[
\begin{tikzcd} %\arrow[r,"popis"], 
H_1 \arrow[dr]\arrow[r]  & H_2 \arrow[r]\arrow[dr] & H_3 \arrow[r]\arrow[dr] & \cdots
\\
E_{S_\infty} \arrow[u, leftrightarrow] \arrow[r] \arrow[d, leftrightarrow] & C_2 \arrow[r]\arrow[u] & C_3 \arrow[r]\arrow[u] & \cdots
\\
R_1   \arrow[ur]\arrow[r] & R_2\arrow[r] \arrow[ur] & R_3\arrow[r] \arrow[ur] & \cdots
\end{tikzcd}
\]

A topological space $X$ is called \emph{strongly locally homogeneous} ($\mathsf{SLH}$) if for every point $x \in X$ there is an arbitrarily small neighborhood $U$ of $x$ such that for every $y\in U$ there exists an onto homeomorphism $h \colon X \to X$ satisfying $h(x) = y$ and $h \restriction_{X \setminus U} = \mathrm{id}_{X \setminus U}$. A separable topological space $X$ is called \emph{countable dense homogeneous} ($\mathsf{CDH}$) if for every two countable dense sets $C, D\subseteq X$ there exists an onto homeomorphism $h\colon X\to X$ such that $h(C)=D$. A separable topological space $X$ is said to be \emph{strongly countable dense homogeneous} ($\mathsf{SCDH}$) if for every open cover $\mathcal U$ of $X$ and every two countable dense sets $C, D\subseteq X$ there is an onto homeomorphism $h\colon X\to X$ such that $h(C)=D$ and $h$ is limited by $\mathcal U$, i.e. for every point $x \in X$ the set $\{x, h(x)\}$ is covered by an element of $\mathcal U$.

It is a well known fact \cite[Theorem 1.6.9]{vanMill} that every  Polish locally compact $\mathsf{SLH}$ space is $\mathsf{CDH}$.
In fact, as noted e.g. in \cite{AndersonCurtisvanMill}, every Polish locally compact $\mathsf{SLH}$ space is $\mathsf{SCDH}$
(it follows e.g. by a careful reading of the proof of \cite[Theorem 1.6.9]{vanMill}.)
Since every open subset of $\R^n$ is $\mathsf{SLH}$ we deduce that it is also $\mathsf{SCDH}$. 

\begin{lemma}\label{SCDH}
Let $n \in \N$. Assume that $G \subseteq \R^n$ is a nonempty open set and $D_1 , D_2$ are countable dense subsets of $G$. Then there exists an onto homeomorphism $h \colon \R^n \to \R^n$ such that $h(D_1)=D_2$ and $h(x)=x$ for every $x \in \R^n \setminus G$.
\end{lemma}
\begin{proof}
If $G = \R^n$, we are done, as $\R^n$ is $\mathsf{CDH}$. Hence, assume that $G \neq \R^n$. For every $x \in G$ let
\[ U(x) := \Big\lbrace y \in \R^n \, ; \ \lVert x-y \rVert < \frac{1}{2} \mathrm{dist} (x , \, \R^n \setminus G) \Big\rbrace . \]
Clearly, the family $\mathcal{U} := \big\lbrace U(x) \, ; \, x \in G \big\rbrace$ is an open cover of $G$. As an open subset of $\R^n$, the set $G$ is $\mathsf{SCDH}$. Hence, there is an onto homeomorphism $f \colon G \to G$ such that $f(D_1) = D_2$ and $f$ is limited by $\mathcal{U}$. Define the desired mapping $h \colon \R^n \to \R^n$ by
\begin{equation*}
    h(x) = \begin{cases}
    f(x) & \textup{if } x \in G\\
    x & \textup{if } x \in \R^n \setminus G.
            \end{cases}
\end{equation*}
Clearly, $h(D_1)=D_2$ and $h(x)=x$ for every $x \in \R^n \setminus G$. It is also clear that $h$ is a bijection and both of the mappings $h$, $h^{-1}$ are continuous at each point of the set $\R^n \setminus \partial G$. It remains to prove that they are continuous also at every point of the set $\partial G$.
\begin{claim}
For every $z \in \partial G$ and every $y \in G$ we have $\lVert z-h(y) \rVert < 3 \lVert z-y \rVert$ and $\big\lVert z-h^{-1}(y) \big\rVert < 3 \lVert z-y \rVert$.
\end{claim}
\begin{claimproof}
Let $z \in \partial G$ and $y \in G$ be arbitrary. By the definition of $h$, we have $h(y)=f(y)$, $h^{-1}(y)=f^{-1}(y) \in G$ and $f \big( h^{-1}(y) \big) = y$. Hence, since $f$ is limited by $\mathcal{U}$, there exist points $x_1 , x_2 \in G$ such that $\{ y , h(y) \} \subseteq U(x_1)$ and $\{ y , h^{-1}(y) \} \subseteq U(x_2)$. It follows that
\begin{align*}
    \lVert x_1-y \rVert &< \frac{1}{2} \mathrm{dist} (x_1 , \, \R^n \setminus G), & \lVert x_1-h(y) \rVert &< \frac{1}{2} \mathrm{dist} (x_1 , \, \R^n \setminus G),\\
    \lVert x_2-y \rVert &< \frac{1}{2} \mathrm{dist} (x_2 , \, \R^n \setminus G), & \lVert x_2-h^{-1}(y) \rVert &< \frac{1}{2} \mathrm{dist} (x_2 , \, \R^n \setminus G).
\end{align*}
Moreover, for every $i \in \{ 1,2 \}$ we have
\[ \mathrm{dist} (x_i , \, \R^n \setminus G) \leq \lVert x_i - z \rVert \leq \lVert x_i - y \rVert + \lVert y - z \rVert < \frac{1}{2} \mathrm{dist} (x_i , \, \R^n \setminus G) + \lVert y - z \rVert , \]
which implies that $\frac{1}{2} \mathrm{dist} (x_i , \, \R^n \setminus G) < \lVert y - z \rVert$. Therefore,
\begin{align*}
    \lVert z-h(y) \rVert &\leq \lVert z-y \rVert + \lVert y-x_1 \rVert + \lVert x_1-h(y) \rVert\\
    &\leq \lVert z-y \rVert + \frac{1}{2} \mathrm{dist} (x_1 , \, \R^n \setminus G) + \frac{1}{2} \mathrm{dist} (x_1 , \, \R^n \setminus G)\\
    &< \lVert z-y \rVert + \lVert y - z \rVert + \lVert y - z \rVert = 3 \lVert z-y \rVert
\end{align*}
and
\begin{align*}
    \lVert z-h^{-1}(y) \rVert &\leq \lVert z-y \rVert + \lVert y-x_2 \rVert + \lVert x_2-h^{-1}(y) \rVert\\
    &\leq \lVert z-y \rVert + \frac{1}{2} \mathrm{dist} (x_2 , \, \R^n \setminus G) + \frac{1}{2} \mathrm{dist} (x_2 , \, \R^n \setminus G)\\
    &< \lVert z-y \rVert + \lVert y - z \rVert + \lVert y - z \rVert = 3 \lVert z-y \rVert .
\end{align*}\claimend
\end{claimproof}
Now we are ready to finish the proof. Let $z \in \partial G$ and $\varepsilon > 0$ be given and define $\delta := \frac{1}{3} \varepsilon$. We claim that $\lVert h(z) - h(y) \rVert < \varepsilon$ and $\big\lVert h^{-1}(z) - h^{-1}(y) \big\rVert < \varepsilon$ for every $y \in \R^n$ with $\lVert z-y \rVert < \delta$. Let $y \in \R^n$ with $\lVert z-y \rVert < \delta$ be arbitrary. If $y \notin G$, we immediately receive
\[ \lVert h(z) - h(y) \rVert = \lVert z - y \rVert < \delta = \frac{\varepsilon}{3} < \varepsilon \]
and
\[ \big\lVert h^{-1}(z) - h^{-1}(y) \big\rVert = \lVert z - y \rVert < \delta = \frac{\varepsilon}{3} < \varepsilon . \]
On the other hand, if $y \in G$, we have
\[ \lVert h(z) - h(y) \rVert = \lVert z - h(y) \rVert < 3 \lVert z - y \rVert < 3 \delta = \varepsilon \]
and
\[ \big\lVert h(z) - h^{-1}(y) \big\rVert = \big\lVert z - h^{-1}(y) \big\rVert < 3 \lVert z - y \rVert < 3 \delta = \varepsilon . \]
\end{proof}

Proposition \ref{uniformization}, which follows, is %proved e.g. in \cite[Proposition 10]{KrupskiVejnar} and it is 
a simple enhancement of the Arsenin-Kunugui selection theorem \cite[Theorem 35.46]{Kechris}.

\begin{proposition}\label{uniformization}
Let $X$ be a standard Borel space, $Y$ a Polish space and $B\subseteq X\times Y$ a Borel set whose nonempty vertical sections are infinite and $\sigma$-compact. Then $\pi_X(B)$ is Borel and there exist Borel measurable mappings $f_k\colon \pi_X(B)\to Y$, $k \in \N$, such that:
\begin{enumerate}[label=(\roman*),font=\textup]
    \item $\forall k \in \N \ \forall x \in \pi_X(B) : \big( x , f_k (x) \big) \in B ;$
    \item $\forall x\in \pi_X(B) : \{f_k(x) \, ; \, k\in\N\} \textup{ is dense in the vertical section } B_x ;$
    \item $\forall x\in \pi_X(B) : \textup{the sequence } \big( f_k(x) \big)_{k=1}^{\infty} \textup{ is injective} .$
\end{enumerate}
\end{proposition}
\begin{proof}
First note that the Arsenin-Kunugui selection theorem immediately gives us that $\pi_X(B)$ is Borel. Let $(U_k)_{k=1}^{\infty}$ be a sequence of open subsets of $Y$ such that $U_1 = Y$ and $\{ U_k ; \, k \in \N \}$ is a base of $Y$. For every $k \in \N$ let $B_k := B \cap (X \times U_k)$. Then each of the sets $B_k$, $k \in \N$, is Borel and has $\sigma$-compact vertical sections. Hence, by the Arsenin-Kunugui selection theorem, for every $k \in \N$ the set $X_k := \pi_X (B_k)$ is Borel in $X$ and there is a Borel uniformization $g_k \colon X_k \to Y$ of $B_k$. For every $k \in \N$ define a mapping $\widetilde{g}_k \colon \pi_X(B) \to Y$ by
\begin{equation*}
    \widetilde{g}_k (x) = \begin{cases}
    g_k (x) & \textup{if } x \in X_k \\
    g_1 (x) & \textup{if } x \in \pi_X(B) \setminus X_k .
            \end{cases}
\end{equation*}
Clearly, each of the mappings $\widetilde{g}_k$, $k \in \N$, is Borel measurable and the set $\{\widetilde{g}_k (x) \, ; \, k\in\N\}$ is dense in $B_x$ for every $x \in \pi_X(B)$. Moreover, since $B_x$ is infinite for every $x \in \pi_X(B)$, the mapping $\mu \colon \N \times \pi_X(B) \to \N$ given by
\[ \mu (k,x) = \mathrm{min} \big\lbrace m \in \N \, ; \, \big| \{ \widetilde{g}_i (x) \, ; \, i \in \N , \, i \leq m \} \big| = k \big\rbrace\]
is well-defined. Finally, define the desired mappings $f_k$, $k \in \N$, by
\[ f_k (x) = \widetilde{g}_{\mu (k,x)} (x) \, , \ \ \ x \in \pi_X(B) \, , \ k \in \N . \]
Clearly, the conditions (i), (ii), (iii) are satisfied and it is straightforward to verify that each of the mappings $f_k$, $k \in \N$, is Borel measurable.
\end{proof}
Before moving to the next lemma, let us note that whenever $A$ is a subset of a Euclidean space, we denote by $\overline{A}$ the closure of $A$ with respect to the whole space, even if $A$ is naturally seen as a subset of a certain subspace of the Euclidean space. The same applies to the boundary of $A$ (denoted as $\partial A$) and the interior of $A$ (denoted as $\Int (A)$).
\begin{lemma}\label{BorelSets}
Let $n \in \N$. Then each of the following three sets is Borel in $\mathcal{K} \big( [0,1]^n \big) \times [0,1]^n \colon$
\begin{align*}
    &\Big\lbrace (K,x) \in \mathcal{K} \big( [0,1]^n \big) \times [0,1]^n \, ; \ x \in (0,1)^n \cap \partial K \Big\rbrace ,\\
    &\Big\lbrace (K,x) \in \mathcal{K} \big( [0,1]^n \big) \times [0,1]^n \, ; \ x \in (0,1)^n \setminus K \Big\rbrace ,\\
    &\Big\lbrace (K,x) \in \mathcal{K} \big( [0,1]^n \big) \times [0,1]^n \, ; \ x \in \Int (K) \Big\rbrace .
\end{align*}
\end{lemma}
\begin{proof}
It is evident that the three sets form a partition of $\mathcal{K} \big( [0,1]^n \big) \times (0,1)^n$, which obviously is a Borel subset of $\mathcal{K} \big( [0,1]^n \big) \times [0,1]^n$. Thus, it suffices to show that at least two of the three sets are Borel. Let us start by proving that the first one is Borel. By Lemma \ref{BorelMeas}, the mapping from $\mathcal{K} ( \R^n )$ to $\mathcal{K} ( \R^n )$ given by $K \mapsto \partial K$ is Borel measurable. Hence, by \cite[Theorem 28.8]{Kechris}, the set $\big\lbrace (K,x) \in \mathcal{K} ( \R^n ) \times \R^n \, ; \ x \in \partial K \big\rbrace$ is Borel in $\mathcal{K} ( \R^n ) \times \R^n$, which implies that the set $\big\lbrace (K,x) \in \mathcal{K} \big( [0,1]^n \big) \times [0,1]^n \, ; \ x \in \partial K \big\rbrace$ is Borel in $\mathcal{K} \big( [0,1]^n \big) \times [0,1]^n$. Therefore, as the set $\mathcal{K} \big( [0,1]^n \big) \times (0,1)^n$ is Borel in $\mathcal{K} \big( [0,1]^n \big) \times [0,1]^n$, so is the set
\begin{gather*}
    \Big\lbrace (K,x) \in \mathcal{K} \big( [0,1]^n \big) \times [0,1]^n \, ; \ x \in \partial K \Big\rbrace \cap \Big( \mathcal{K} \big( [0,1]^n \big) \times (0,1)^n \Big)\\
    = \Big\lbrace (K,x) \in \mathcal{K} \big( [0,1]^n \big) \times [0,1]^n \, ; \ x \in (0,1)^n \cap \partial K \Big\rbrace .
\end{gather*}

Now, let us show that the second of the three sets in question is Borel. In fact, we claim that it is open: Let $(K,x) \in \mathcal{K} \big( [0,1]^n \big) \times [0,1]^n$ with $x \in (0,1)^n \setminus K$ be given. Clearly, there exists an open set $U$ in $\R^n$ such that $x \in U \subseteq (0,1)^n$ and $\overline{U} \cap K = \emptyset$. By the definition of the Vietoris topology, the set $G := \big\lbrace L \in \mathcal{K} \big( [0,1]^n \big) \, ; \ L \subseteq [0,1]^n \setminus \overline{U} \big\rbrace$ is open in $\mathcal{K} \big( [0,1]^n \big)$. Consequently, $G \times U$ is open in $\mathcal{K} \big( [0,1]^n \big) \times [0,1]^n$. Moreover, it is obvious that $(K,x) \in G \times U$ and for every $(L,y) \in G \times U$ we have $y \in (0,1)^n \setminus L$.
\end{proof}

\begin{lemma}[{\cite[Exercise 4.29]{Kechris}}]\label{UnionIsCts}
Let $X$ be a Polish space and $k \in \N$. Then the mapping from $\big( \mathcal{K} (X) \big)^k$ to $\mathcal{K} (X)$ given by $\displaystyle (K_1 , \dotsc , K_k) \mapsto \bigcup_{i=1}^k K_i$ is continuous.
\end{lemma}

\begin{theorem}\label{AmbientHomeos}
The equivalence relation of ambient homeomorphisms of compacta in $[0,1]^n$ is Borel reducible to the homeomorphism equivalence relation of $n$-dimensional continua in $[0,1]^{n+1}$ for every $n\in\N$. In particular, $R_n\leq_B C_{n+1}$ for every $n\in\N$.
\end{theorem}

\begin{proof}
Let $n\in\N$. By \cite[Exercise~4.13]{Hjorth} (see \cite{ChangGaoDim} for more details), the equivalence relation $R_1$ is Borel bireducible with the isomorphism ER of countable graphs, which is (as was explained at the beginning of the proof of Theorem \ref{thmAR2}) Borel bireducible with the homeomorphism ER of dendrites in $[0,1]^2$. Therefore, as dendrites are known to be $1$-dimensional and every dendrite is a continuum, the relation $R_1$ is Borel reducible to the homeomorphism ER of $1$-dimensional continua in $[0,1]^2$. Thus, let us assume that $n > 1$. Denote $X:=\mathcal{K} \big( [0,1]^n \big)$ and $Y:=[0,1]^n$. Then $X$, $Y$ are Polish spaces. Consider the sets
\begin{align*}
    B_1 &:= \big\lbrace (K,y) \in X \times Y ; \ y \in (0,1)^n \setminus K \big\rbrace ,\\
    B_2 &:= \big\lbrace (K,y) \in X \times Y ; \ y \in \Int (K) \big\rbrace .
\end{align*}
By Lemma \ref{BorelSets}, these two sets are Borel. Moreover, for every $K \in X$ we have
\begin{align*}
    \big\lbrace y \in Y ; \ (K,y) \in B_1 \big\rbrace &= \big\lbrace y \in [0,1]^n ; \ y \in (0,1)^n \setminus K \big\rbrace = (0,1)^n \setminus K ,\\
    \big\lbrace y \in Y ; \ (K,y) \in B_2 \big\rbrace &= \big\lbrace y \in [0,1]^n ; \ y \in \Int (K) \big\rbrace = \Int (K) .
\end{align*}
In particular, vertical sections of $B_1$ and $B_2$ are open in $\R^n$ and thus they are $\sigma$-compact and infinite (when nonempty). Hence, by Proposition \ref{uniformization}, the sets $\pi_X(B_1)$, $\pi_X(B_2)$ are Borel and there exist Borel measurable mappings $f_k^1 \colon \pi_X(B_1) \to Y$, $f_k^2 \colon \pi_X(B_2) \to Y$, $k \in \N$, such that: 
\begin{itemize}
    \item $\forall \, k \in \N \ \ \forall \, K \in \pi_X(B_1) : \, f_k^1 (K) \in (0,1)^n \setminus K \, ;$
    \item $\forall \, k \in \N \ \ \forall \, K \in \pi_X(B_2) : \, f_k^2 (K) \in \Int (K) \, ;$
    \item the set $\{ f_k^1 (K) \, ; \, k\in\N\}$ is dense in $(0,1)^n \setminus K$ for every $K \in \pi_X(B_1) \, ;$
    \item the set $\{ f_k^2 (K) \, ; \, k\in\N\}$ is dense in $\Int (K)$ for every $K \in \pi_X(B_2) \, ;$
    \item the sequence $\big( f_k^i (K) \big)_{k=1}^{\infty}$ is injective for all $i \in \{ 1,2 \}$ and $K \in \pi_X(B_i)$.
\end{itemize}
It immediately follows from the definitions of the sets $B_1$ and $B_2$ that $\pi_X(B_1) = \big\lbrace K \in X \, ; \ K \neq [0,1]^n \big\rbrace$ and $\pi_X(B_2) = \big\lbrace K \in X \, ; \ \Int (K) \neq \emptyset \big\rbrace$. For every $i \in \{ 1,2 \}$, $K \in \pi_X(B_i)$ and $m \in \N$ with $m > 1$ let
\[ {\delta}_m^i (K) := \mathrm{min} \Big\lbrace \big\lVert f_k^i (K) - f_l^i (K) \big\rVert \, ; \ k,l \in \N , \, k \leq m , \, l \leq m , \, k \neq l \Big\rbrace \]
and ${\delta}_1^i (K) := 1$. For every $K \in \pi_X(B_1)$ and every $m \in \N$ let
\[ r_m^1 (K) := \mathrm{max} \Big\lbrace 2^{-k} ; \, k \in \N , \, \forall x \in K : \norm{x-f_m^1 (K)} \geq 2^{-k} \Big\rbrace . \]
Also, for every $K \in \pi_X(B_2)$ and every $m \in \N$ let
\[ r_m^2 (K) := \mathrm{max} \Big\lbrace 2^{-k} ; \, k \in \N , \, \forall x \in [0,1]^n \setminus \Int (K) : \norm{x-f_m^2 (K)} \geq 2^{-k} \Big\rbrace . \]
For every $i \in \{ 1,2 \}$, every $K \in \pi_X(B_i)$ and every $m \in \N$ let
\[ t_m^i (K) := \frac{1}{2} \mathrm{min} \Big\lbrace 2^{-m} , \ \frac{r_m^i (K)}{\sqrt{n}} , \ \frac{{\delta}_m^i (K)}{{\delta}_m^i (K) + \sqrt{n}} \Big\rbrace . \]
It is clear that each of the numbers ${\delta}_m^i (K)$, $r_m^i (K)$ and $t_m^i (K)$ is positive for every $i \in \{ 1,2 \}$, every $K \in \pi_X(B_i)$ and every $m \in \N$. Choose and fix three distinct points $p_1 , p_2 , p_3 \in [0,1]^n$. For the sake of simpler notation let us identify ${\R}^{n+1}$ with ${\R}^n \times \R$. For every $i \in \{ 1,2 \}$, $j \in \{ 1,2,3 \}$, $K \in \pi_X(B_i)$, $m \in \N$ and $t \in \R$ let
\begin{align*}
    p_j^i (K,m,t) :&= \big( f_m^i (K) , 0 \big) + t (p_j - f_m^i (K) , 1)\\
    &= \big( t p_j + (1-t) f_m^i (K) , t \big) .
\end{align*}
Moreover, for every $i \in \{ 1,2 \}$, $K \in X$ and $m \in \N$ let
\begin{equation*}
    P_m^i (K) := \begin{cases}
     \big\lbrace p_j^i (K,m,t) \, ; \ j \leq i+1 , \ 0 < t \leq t_m^i (K) \big\rbrace & \textup{if } K \in \pi_X(B_i) \\
    \emptyset & \textup{if } K \notin \pi_X(B_i) .
\end{cases}
\end{equation*}
For every $i \in \{ 1,2 \}$ and every $K \in X$ denote
\[ P^i (K) := \bigcup_{m=1}^{\infty} P_m^i (K) . \]

\begin{claim}\label{DiamGoesToZero3}
For every $i \in \{ 1,2 \}$ and $K \in \pi_X(B_i)$ we have
\[ \lim_{m \to \infty} \mathrm{diam} \big( P_m^i (K) \big) = 0 . \]
\end{claim}
\begin{claimproof}
Let $i \in \{ 1,2 \}$ and $K \in \pi_X(B_i)$ be given. For every $m \in \N$ and every $x \in P_m^i (K)$ there is a natural number $j \leq i+1$ and a real number $t$ with $0 < t \leq t_m^i (K)$ such that $x=p_{j}^i (K,m,t)$, hence
\begin{align*}
    \big\lVert x- \big( f_m^i (K) , 0 \big) \big\rVert &= \big\lVert \big( t p_j -t f_m^i (K) , t \big) \big\rVert = t \big\lVert \big( p_j - f_m^i (K) , 1 \big) \big\rVert \\
    &\leq t \big( \lVert p_j \rVert + \lVert f_m^i (K) \rVert +1 \big) \leq t \big( \sqrt{n}+\sqrt{n}+1 \big) \\
    &\leq t_m^i (K) \big( 1+2 \sqrt{n} \big) \leq 2^{-m-1} \big( 1+2 \sqrt{n} \big).
\end{align*}
Therefore, for every $m \in \N$ and every two points $x_1 , x_2 \in P_m^i (K)$ we have
\[ \lVert x_1 - x_2 \rVert \leq \big\lVert x_1 - \big( f_m^i (K) , 0 \big) \big\rVert + \big\lVert x_2 - \big( f_m^i (K) , 0 \big) \big\rVert \leq 2^{-m} \big( 1+2 \sqrt{n} \big) . \]
Thus, $\mathrm{diam} \big( P_m^i (K) \big) \leq 2^{-m} \big( 1+2 \sqrt{n} \big)$, $m \in \N$, and the claim follows. \claimend
\end{claimproof}

\begin{claim}\label{MarkersAreDisjoint1}
For every $i \in \{ 1,2 \}$, $K \in X$ and every two distinct numbers $m_1 , m_2 \in \N$ we have $P_{m_1}^i (K) \cap P_{m_2}^i (K) = \emptyset$.
\end{claim}
\begin{claimproof}
Let $i \in \{ 1,2 \}$, $K \in X$ and $m_1 , m_2 \in \N$ with $m_1 \neq m_2$ be given and assume that $P_{m_1}^i (K) \cap P_{m_2}^i (K) \neq \emptyset$. Then, by the definitions of the sets $P_{m_1}^i (K)$ and $P_{m_2}^i (K)$, it is clear that $K \in \pi_X(B_i)$ and there are numbers $j_1 , j_2 \in \N$ and $t_1 , t_2 \in \R$ such that $j_1 \leq i+1$, $j_2 \leq i+1$, $0 < t_1 \leq t_{m_1}^i (K)$, $0 < t_2 \leq t_{m_2}^i (K)$ and
\[ \big( f_{m_1}^i (K) , 0 \big) + t_1 \big( p_{j_1} - f_{m_1}^i (K) , 1 \big) = \big( f_{m_2}^i (K) , 0 \big) + t_2 \big( p_{j_2} - f_{m_2}^i (K) , 1 \big) . \]
This equality clearly implies that $t_1 = t_2$ and
\[ f_{m_1}^i (K) + t p_{j_1} - t f_{m_1}^i (K) = f_{m_2}^i (K) + t p_{j_2} - t f_{m_2}^i (K), \]
where $t:=t_1$. It follows that
\[ t \big( p_{j_1} - p_{j_2} - f_{m_1}^i (K) + f_{m_2}^i (K) \big) = f_{m_2}^i (K) - f_{m_1}^i (K). \]
In particular,
\[ t \big\lVert p_{j_1} - p_{j_2} - f_{m_1}^i (K) + f_{m_2}^i (K) \big\rVert = \big\lVert f_{m_2}^i (K) - f_{m_1}^i (K) \big\rVert . \]
Since $m_1 \neq m_2$, we have $f_{m_1}^i (K) \neq f_{m_2}^i (K)$, which shows that
\[ \big\lVert f_{m_2}^i (K) - f_{m_1}^i (K) \big\rVert \neq 0. \]
Hence,
\[ \big\lVert p_{j_1} - p_{j_2} - f_{m_1}^i (K) + f_{m_2}^i (K) \big\rVert \neq 0 \]
and we have
\[ t = \frac{\big\lVert f_{m_2}^i (K) - f_{m_1}^i (K) \big\rVert}{\big\lVert p_{j_1} - p_{j_2} - f_{m_1}^i (K) + f_{m_2}^i (K) \big\rVert} . \]
By the triangle inequality,
\[ t \geq \frac{\big\lVert f_{m_2}^i (K) - f_{m_1}^i (K) \big\rVert}{\big\lVert f_{m_2}^i (K) - f_{m_1}^i (K) \big\rVert + \big\lVert p_{j_2} - p_{j_1} \big\rVert} . \]
Let $m:= \mathrm{max} \{ m_1 , m_2 \}$. Then $m > 1$ and it follows from the definition of ${\delta}_m^i (K)$ that ${\delta}_m^i (K) \leq \big\lVert f_{m_2}^i (K) - f_{m_1}^i (K) \big\rVert$. Therefore, as the function $s \mapsto \frac{s}{s+c}$ is non-decreasing on $(0 , \infty)$ for every $c \geq 0$, we have
\[ t \geq \frac{{\delta}_m^i (K)}{{\delta}_m^i (K) + \big\lVert p_{j_2} - p_{j_1} \big\rVert} . \]
Moreover, since $p_{j_1} , p_{j_2} \in [0,1]^n$, we have $\big\lVert p_{j_2} - p_{j_1} \big\rVert \leq \mathrm{diam} \big( [0,1]^n \big) = \sqrt{n}$. Hence,
\[ t \geq \frac{{\delta}_m^i (K)}{{\delta}_m^i (K) + \sqrt{n} } . \]
Finally, it follows from the definition of $t_m^i (K)$ that
\[ \frac{{\delta}_m^i (K)}{{\delta}_m^i (K) + \sqrt{n} } \geq 2 t_m^i (K) , \]
thus, $t \geq 2 t_m^i (K)$, which contradicts the fact that $t \leq t_m^i (K) < 2 t_m^i (K)$.\claimend
\end{claimproof}
\begin{claim}\label{MarkersAreDisjoint2}
For every $K \in X$ we have $P^1 (K) \subseteq \big( [0,1]^n \setminus K \big) \times \big( 0, \frac{1}{4} \big]$ and $P^2 (K) \subseteq \Int (K) \times \big( 0, \frac{1}{4} \big]$. In particular,
\[ P^1 (K) \cap \big( \partial K \times [0,1] \big) = P^2 (K) \cap \big( \partial K \times [0,1] \big) = P^1 (K) \cap P^2 (K) = \emptyset \]
for every $K \in X$.
\end{claim}

\begin{claimproof}
Let $K \in X$ be arbitrary and denote $A_1 := [0,1]^n \setminus K$, $A_2 := \Int (K)$. We will show that $P^i (K) \subseteq A_i \times \big( 0, \frac{1}{4} \big]$ for every $i \in \{ 1,2 \}$. Let $i \in \{ 1,2 \}$ be given. If $K \notin \pi_X(B_i)$, we have $P^i (K) = \emptyset$ and we are done. Thus, let us assume that $K \in \pi_X(B_i)$ and let $p \in P^i (K)$ be given. By the definition of $P^i (K)$, there exist $m \in \N$, $j \in \N$ with $j \leq i+1$ and $t > 0$ with $t \leq t_m^i (K)$ such that $p = \big( tp_j + (1-t) f_m^i (K) , \, t  \big)$. Since $t_m^i (K) \leq \frac{1}{2} \cdot 2^{-m} \leq \frac{1}{4}$, we have $0 < t \leq \frac{1}{4}$. Hence, it suffices to show that the point $x:= tp_j + (1-t) f_m^i (K)$ lies in $A_i$. As a convex combination of points belonging to $[0,1]^n$, the point $x$ belongs to $[0,1]^n$ as well. Moreover, we have
\begin{align*}
    \big\lVert x - f_m^i (K) \big\rVert &= \big\lVert t p_j - t f_m^i (K) \big\rVert = t \big\lVert p_j - f_m^i (K) \big\rVert \leq t \cdot \mathrm{diam} \big( [0,1]^n \big)\\
    &\leq t \cdot \sqrt{n} \leq t_m^i (K) \cdot \sqrt{n} \leq \frac{r_m^i (K)}{2 \sqrt{n}} \cdot \sqrt{n} = \frac{r_m^i (K)}{2} < r_m^i (K) .
\end{align*}
Thus, by the definitions of $A_i$ and $r_m^i (K)$, we conclude that $x \in A_i$.\claimend
\end{claimproof}

\begin{claim}
The set $\big( [0,1]^n \times \{ 0 \} \big) \cup \big( \partial K \times (0,1] \big) \cup P^1 (K) \cup P^2 (K)$ is an $n$-dimensional subcontinuum of $[0,1]^{n+1}$ for every $K \in X$.
\end{claim}
\begin{claimproof}
Let $K \in X$ be given and denote
\[ Z:= \big( [0,1]^n \times \{ 0 \} \big) \cup \big( \partial K \times (0,1] \big) \cup P^1 (K) \cup P^2 (K) . \]
It is easy to see that $Z$ is a path-connected subset of $[0,1]^{n+1}$. Let us show that $Z$ is closed (in $\R^{n+1}$). Since the set $\big( [0,1]^n \times \{ 0 \} \big) \cup \big( \partial K \times (0,1] \big)$ clearly is closed, it suffices to show that the set $\big( [0,1]^n \times \{ 0 \} \big) \cup P^i (K)$ is closed for each $i \in \{ 1,2 \}$. Let $i \in \{ 1,2 \}$ be arbitrary. We can assume that $K \in \pi_X(B_i)$, otherwise $P^i (K) = \emptyset$ and we are done. It is clear that
\[ \overline{P_m^i (K)} = P_m^i (K) \cup \big\lbrace \big( f_m^i (K) , 0 \big) \big\rbrace \]
for every $m \in \N$. In particular, we have $\overline{P_m^i (K)} \setminus P_m^i (K) \subseteq [0,1]^n \times \{ 0 \}$ and $\overline{P_m^i (K)} \cap \big( [0,1]^n \times \{ 0 \} \big) \neq \emptyset$ for every $m \in \N$. Therefore, using Claim \ref{DiamGoesToZero3} and the closedness of the set $[0,1]^n \times \{ 0 \}$, we conclude by Lemma \ref{ClosednessOfUnion} that the set $\big( [0,1]^n \times \{ 0 \} \big) \cup P^i (K)$ is closed.

As a closed bounded subset of $\R^{n+1}$, the set $Z$ is compact. Also, since $[0,1]^n \times \{ 0 \} \subseteq Z$, it is clear that the dimension of $Z$ is at least $n$. Finally, since one can easily see that $Z$ is meagre in $\R^{n+1}$, the interior of $Z$ is empty. Hence, by \cite[Corollary 1.8.4, Theorem 1.8.11]{EngelkingDimension}, the dimension of $Z$ is equal to $n$.
\claimend
\end{claimproof}

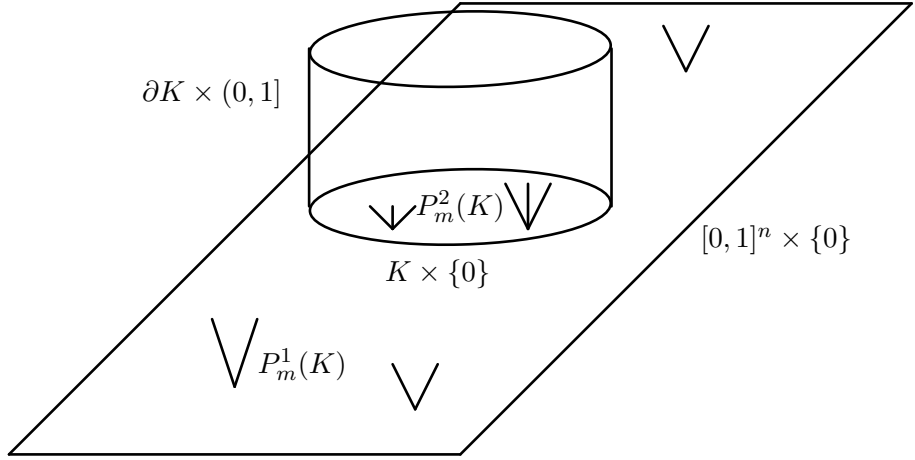
\begin{figure}
\begin{tikzpicture}[line cap=round,line join=round,>=triangle 45,x=3cm,y=3cm]
\draw [line width=1pt] (0,0)-- (2,0);  % podlaha
\draw [line width=1pt] (2,2)-- (4,2);
\draw [line width=1pt] (2,2)-- (0,0);
\draw [line width=1pt] (4,2)-- (2,0);

\draw [line width=1pt] (0.9,0.6)-- (1,0.3);  % dvojnozka vlevo
\draw [line width=1pt] (1.1,0.6)-- (1,0.3);

\draw [line width=1pt] (1.7,0.4)-- (1.8,0.2); % dvojnozka vpravo dole
\draw [line width=1pt] (1.8,0.2)-- (1.9,0.4);

\draw [line width=1pt] (2.9,1.9)-- (3,1.7); % dvojnozka vpravo nahore
\draw [line width=1pt] (3,1.7)-- (3.1,1.9);

\draw [line width=1pt] (1.6,1.1)-- (1.7,1);  % trojnozka vlevo
\draw [line width=1pt] (1.7,1)-- (1.7,1.1);
\draw [line width=1pt] (1.7,1)-- (1.8,1.1);
\draw[color=black] (2.0,1.1) node {$P^2_m(K)$};

\draw [line width=1pt] (2.3,1)-- (2.2,1.2); % trojnozka vpravo
\draw [line width=1pt] (2.3,1)-- (2.3,1.2);
\draw [line width=1pt] (2.3,1)-- (2.4,1.2);

\draw [rotate around={1.5:(2.1,1.8)},line width=1pt] (2.0,1.8) ellipse (2cm and 0.5cm);  % horní elipsa
\draw [rotate around={1.5:(2.1,1.1)},line width=1pt] (2.0,1.1) ellipse (2cm and 0.5cm);    % dolní elipsa
\draw [line width=1pt] (1.33,1.1)-- (1.33,1.8);  % svisle cary
\draw [line width=1pt] (2.67,1.1)-- (2.67,1.8);

\draw[color=black] (3.4,0.95) node {$[0,1]^n\times \{0\}$};
\draw[color=black] (1.9,0.8) node {$ K\times\{0\}$};
\draw[color=black] (1.3,0.4) node {$P^1_m(K)$};
\draw[color=black] (0.9,1.6) node {$\partial K\times(0,1]$};
\end{tikzpicture}
\caption{Compact set $\varphi(K)$ topologically.}\label{Fig:fik}
\end{figure}

Define a mapping $\varphi \colon X \to {\mathsf{C}}_n \big( [0,1]^{n+1} \big)$ by
\[ \varphi (K) = \big( [0,1]^n \times \{ 0 \} \big) \cup \big( \partial K \times (0,1] \big) \cup P^1 (K) \cup P^2 (K) \]
(see Figure \ref{Fig:fik}).
By the previous claim, the mapping $\varphi$ is well-defined. Let us show that it is a reduction from $R_n$ to the homeomorphism ER on ${\mathsf{C}}_n \big( [0,1]^{n+1} \big)$. Given $K,L \in X$, we would like to show that $(K,L) \in R_n$ if and only if $\varphi (K)$ is homeomorphic to $\varphi (L)$. First, assume that $(K,L) \in R_n$. Let $g_0$ be a self-homeomorphism of $[0,1]^n$ such that $g_0 (K)=L$. It is clear that $K = [0,1]^n$ if and only if $L = [0,1]^n$. If this is the case, we are immediately done, as $\varphi (K) = \varphi \big( [0,1]^n \big) = \varphi (L)$. Hence, let us assume that $K \neq [0,1]^n \neq L$. In other words, assume that $K,L \in \pi_X(B_1)$. Moreover, let us distinguish two cases:
\begin{claim}
If $K \notin \pi_X(B_2)$, then $\varphi (K)$ is homeomorphic to $\varphi (L)$.
\end{claim}
\begin{claimproof}
The assumption that $K \notin \pi_X(B_2)$ is equivalent to $\Int (K) = \emptyset$. By the domain invariance theorem, $g_0$ maps $\Int (K)$ onto $\Int (L)$. Thus, $\Int (L) = \emptyset$ and $L \notin \pi_X(B_2)$. It follows that $P^2 (K) = \emptyset$, $P^2 (L) = \emptyset$ and
\begin{align*}
    \varphi (K) &= \big( [0,1]^n \times \{ 0 \} \big) \cup \big( \partial K \times (0,1] \big) \cup P^1 (K) , \\
    \varphi (L) &= \big( [0,1]^n \times \{ 0 \} \big) \cup \big( \partial L \times (0,1] \big) \cup P^1 (L) .
\end{align*}
Let $D_1 := g_0 \big( \{ f_k^1 (K) \, ; \, k \in \N \} \big)$ and $D_2 := \{ f_k^1 (L) \, ; \, k \in \N \}$. We know that $D_2$ is a dense subset of $(0,1)^n \setminus L$. By the bijectivity of $g_0$ and by the domain invariance theorem, we have
\[ g_0 \big( (0,1)^n \setminus K \big) = g_0 \big( (0,1)^n \big) \setminus g_0 (K) = (0,1)^n \setminus L . \]
Therefore, since the set $\{ f_k^1 (K) \, ; \, k \in \N \}$ is dense in $(0,1)^n \setminus K$, it follows that $D_1$ is a dense subset of $(0,1)^n \setminus L$. By Lemma \ref{SCDH}, there exists an onto homeomorphism $g_1 \colon \R^n \to \R^n$ such that $g_1 (D_1) = D_2$ and $g_1 (x) = x$ for every $x \in \R^n \setminus \big( (0,1)^n \setminus L \big) = \big( \R^n \setminus (0,1)^n \big) \cup L$. Hence, the mapping $g \colon [0,1]^n \to [0,1]^n$ given by $g(x) = g_1 \big( g_0 (x) \big)$ is a self-homeomorphism of $[0,1]^n$ satisfying $g(K) = L$ and $~{g \big( \{ f_k^1 (K) \, ; \, k \in \N \} \big) = \{ f_k^1 (L) \, ; \, k \in \N \} }$. Also, by the domain invariance theorem, $g \big( \Int (K) \big) = \Int (L)$, which implies that $g (\partial K) = \partial L$. For every $m \in \N$ let $\mu (m)$ be the (uniquely determined) natural number such that $g \big( f_m^1 (K) \big) = f_{\mu (m)}^1 (L)$. Denote $\Lambda := \{ 1,2 \}$ and $I := (0,1]$. Define a mapping $h \colon \varphi (K) \to \varphi (L)$ by
\begin{align*}
    h(x,0) &= \big( g(x) , 0 \big) \, , & x &\in [0,1]^n \, ;\\
    h(x,t) &= \big( g(x) , t \big) \, , & x &\in \partial K , \, t \in I \, ; \\
    h \Big( p_j^1 \big( K , m , s t_m^1 (K) \big) \Big) &= p_j^1 \big( L , \mu (m) , s t_{\mu (m)}^1 (L) \big) \, , & j &\in \Lambda , \, m \in \N , \, s \in I .
\end{align*}
It is easy to see that $h$ is a well-defined bijection satisfying
\[ h \big( [0,1]^n \times \{ 0 \} \big) = [0,1]^n \times \{ 0 \} , \ h \big( \partial K \times I \big) = \partial L \times I , \ h \big( P^1 (K) \big) = P^1 (L) . \]
Moreover, with the help of Claims \ref{DiamGoesToZero3}, \ref{MarkersAreDisjoint1} and \ref{MarkersAreDisjoint2} it is straightforward to verify that $h$ is continuous. Thus, as $\varphi (K)$ is compact, $h$ is a homeomorphism. \claimend
\end{claimproof}

\begin{claim}\label{xxx}
If $K \in \pi_X(B_2)$, then $\varphi (K)$ is homeomorphic to $\varphi (L)$.
\end{claim}
\begin{claimproof}
We will proceed similarly as in the previous case, but we will leave out some of the details this time. Since $K \in \pi_X(B_2)$, we have $\Int (K) \neq \emptyset$. Thus, $\Int (L) = g_0 \big( \Int (K) \big) \neq \emptyset$ and $L \in \pi_X(B_2)$. Let
\begin{align*}
    D_1^1 &:= g_0 \big( \{ f_k^1 (K) \, ; \, k \in \N \} \big) , & D_2^1 &:= \{ f_k^1 (L) \, ; \, k \in \N \} ,\\
    D_1^2 &:= g_0 \big( \{ f_k^2 (K) \, ; \, k \in \N \} \big) , & D_2^2 &:= \{ f_k^2 (L) \, ; \, k \in \N \}.
\end{align*}
Then $D_1^1$, $D_2^1$ are countable dense subsets of $(0,1)^n \setminus L$ and $D_1^2$, $D_2^2$ are countable dense subsets of $\Int (L)$. By Lemma \ref{SCDH}, there are onto homeomorphisms $g_1 , g_2 \colon \R^n \to \R^n$ such that $g_1 (D_1^1) = D_2^1$, $g_2 (D_1^2) = D_2^2$, $g_1 (x) = x$ for every $x \in \big( \R^n \setminus (0,1)^n \big) \cup L$ and $g_2 (x) = x$ for every $x \in \R^n \setminus \Int (L)$. At this point, one can easily see that $g_1 (L) = g_2 (L) = L$, $g_1 \big( D_1^2 \big) = D_1^2$, $g_2 \big( D_2^1 \big) = D_2^1$ and $g_1 \big( [0,1]^n \big) = g_2 \big( [0,1]^n \big) = [0,1]^n$. Consequently, it is clear that the mapping $g \colon [0,1]^n \to [0,1]^n$ given by $~{g(x) = (g_2 \circ g_1 \circ g_0)(x)}$, $x \in [0,1]^n$, is a self-homeomorphism of $[0,1]^n$ such that $g(K)=L$ and $g \big( \{ f_k^i (K) \, ; \, k \in \N \} \big) = \{ f_k^i (L) \, ; \, k \in \N \}$ for every $i \in \{ 1,2 \}$. Moreover, by the domain invariance theorem, $g \big( \Int (K) \big) = \Int (L)$, which implies that $g (\partial K) = \partial L$. For each $m \in \N$ let $\mu (m)$, $\nu (m)$ be the (uniquely determined) natural numbers such that $g \big( f_m^1 (K) \big) = f_{\mu (m)}^1 (L)$ and $g \big( f_m^2 (K) \big) = f_{\nu (m)}^2 (L)$. Denote $\Lambda := \{ 1,2 \}$, $\Gamma := \{ 1,2,3 \}$ and $I := (0,1]$. Let us define a mapping $h \colon \varphi (K) \to \varphi (L)$ by
\begin{align*}
    h(x,0) &= \big( g(x) , 0 \big) \, , & x &\in [0,1]^n \, ;\\
    h(x,t) &= \big( g(x) , t \big) \, , & x &\in \partial K , \, t \in I \, ; \\
    h \Big( p_j^1 \big( K , m , s t_m^1 (K) \big) \Big) &= p_j^1 \big( L , \mu (m) , s t_{\mu (m)}^1 (L) \big) \, , & j &\in \Lambda , \, m \in \N , \, s \in I \, ; \\
    h \Big( p_j^2 \big( K , m , s t_m^2 (K) \big) \Big) &= p_j^2 \big( L , \nu (m) , s t_{\nu (m)}^2 (L) \big) \, , & j &\in \Gamma , \, m \in \N , \, s \in I .
\end{align*}
Again, it is straightforward to verify that $h$ is an onto homeomorphism.
\claimend
\end{claimproof}

Conversely, assuming that $\varphi (K)$ is homeomorphic to $\varphi (L)$, let us show that $(K,L) \in R_n$.

\begin{claim}\label{yyy}
If $\varphi (K)$ is homeomorphic to $\varphi (L)$, then $(K,L) \in R_n$.
\end{claim}

\begin{claimproof}
Let $h \colon \varphi (K) \to \varphi (L)$ be an onto homeomorphism. For every $Z \in X$ and every $i \in \{ 1,2 \}$ denote
\begin{equation*}
    D_Z^i = \begin{cases}
    \big\lbrace f_m^i (Z) \, ; \, m \in \N \big\rbrace & \textup{if } Z \in \pi_X(B_i) \\
    \emptyset & \textup{if } Z \notin \pi_X(B_i).
\end{cases} \ \ 
\end{equation*}
As $n>1$, it is easy to see that the set $D_Z^i \times \{ 0 \}$ is equal to
\[ \big\lbrace w \in \varphi (Z) \, ; \textup{ the set } \varphi (Z) \setminus \{ w \} \textup{ has exactly } i+2 \textup{ connected components}  \big\rbrace \]
for every $Z \in X$ and every $i \in \{ 1,2 \}$. Hence, as $h$ is a homeomorphism,
\[ h \big( D_K^i \times \{ 0 \} \big) = D_L^i \times \{ 0 \} \, , \ \ \ i \in \{ 1,2 \} . \tag{i} \]
Moreover, for every $~{ i \in \{ 1,2 \} }$, every $Z \in \pi_X(B_i)$ and every $m \in \N$ it follows from the construction of the set $\varphi (Z)$ that the set
\[ \varphi (Z) \setminus \Big( P_m^i (Z) \cup \big\lbrace \big( f_m^i (Z) , 0 \big) \big\rbrace \Big) \]
is the only connected component of $~{\varphi (Z) \setminus \big\lbrace \big( f_m^i (Z) , 0 \big) \big\rbrace}$ which is not homeomorphic to $(0,1]$. Thanks to this and (i), it is easy to see that \[ h \big( P^i (K) \big) = P^i (L)  \, , \ \ \ i \in \{ 1,2 \} . \tag{ii}\]
Also, as $D_Z^2$ is dense in $\Int (Z)$ for every $Z \in X$, it follows from (i) that
\[ h \big( \overline{ \Int (K) } \times \{ 0 \} \big) = \overline{ \Int (L) } \times \{ 0 \} . \tag{iii} \]
Since $D_Z^1 \cup D_Z^2$ is dense in $\big( (0,1)^n \setminus Z \big) \cup \Int (Z)$ and $\big( (0,1)^n \setminus Z \big) \cup \Int (Z)$ is dense in $[0,1]^n$ for every $Z \in X$, we have by (i) that
\[ h \big( [0,1]^n \times \{ 0 \} \big) = [0,1]^n \times \{ 0 \} . \tag{iv} \]
Combining (ii) and (iv), we receive
\[ h \Big( \big( [0,1]^n \times \{ 0 \} \big) \cup P^1 (K) \cup P^2 (K) \Big) = \big( [0,1]^n \times \{ 0 \} \big) \cup P^1 (L) \cup P^2 (L), \]
which implies (as $h$ is a bijection) that
\[ h \big( \partial K \times (0,1] \big) = \partial L \times (0,1] . \]
Therefore, as $h$ is a homeomorphism,
\[ h \big( \partial K \times [0,1] \big) = \partial L \times [0,1] . \]
This equality combined with (iv) yields
\[ h \big( \partial K \times \{ 0 \} \big) = \partial L \times \{ 0 \} . \tag{v} \]
Finally, since $\overline{\Int (K)} \cup \partial K = K$ and $\overline{\Int (L)} \cup \partial L = L$, we conclude by (iii) and (v) that
\[ h \big( K \times \{ 0 \} \big) = L \times \{ 0 \} . \tag{vi} \]
Clearly, it follows from (iv) and (vi) that $(K,L) \in R_n$.\claimend
\end{claimproof}

It remains to prove that $\varphi$ is Borel measurable.

\begin{claim}\label{zzz}
The mapping $\varphi$ is Borel measurable.
\end{claim}
\begin{claimproof}
For every $k \in \N$ define a mapping $\varphi_k \colon X \to \mathcal{K} \big( [0,1]^{n+1} \big)$ by
\[ \varphi_k (K) = \big( [0,1]^n \times \{ 0 \} \big) \cup \big( \partial K \times (0,1] \big) \cup \bigcup_{i=1}^{2} \bigcup_{m=1}^{k} P_m^i (K) \, , \ \ \ K \in X . \]
It is easy to see that $\varphi$ is the pointwise limit of the sequence $( \varphi_k )_{k=1}^{\infty}$. Hence, it suffices to prove that each of the mappings $\varphi_k$, $k \in \N$, is Borel measurable. Let $k \in \N$ be arbitrary. For every $i \in \{ 1,2 \}$ and every $m \in \N$ with $m \leq k$ define a mapping $\psi_m^i \colon X \to \mathcal{K} \big( [0,1]^{n+1} \big)$ by
\[ \psi_m^i (K) = \overline{P_m^i (K)} \, , \ \ \ K \in X . \]
Also, define a mapping $\psi \colon X \to \mathcal{K} \big( [0,1]^{n+1} \big)$ by
\[ \psi (K) = \partial K \times [0,1] \, , \ \ \ K \in X . \]
Clearly, we have
\[ \varphi_k (K) = \big( [0,1]^n \times \{ 0 \} \big) \cup \psi (K) \cup \bigcup_{i=1}^{2} \bigcup_{m=1}^{k} \psi_m^i (K) \]
for every $K \in X$. Hence, by Lemma \ref{UnionIsCts}, it is enough to show that each of the mappings $\psi$ and $\psi_m^i$, $i \in \{ 1,2 \}$, $m \in \{ 1 , \dotsc , k \}$, is Borel measurable. It is clear that the mapping from $\mathcal{K} \big( [0,1]^n \big)$ to $\mathcal{K} \big( [0,1]^{n+1} \big)$ given by $K \mapsto K \times [0,1]$ is continuous. Hence, by Lemma \ref{BorelMeas}, the mapping $\psi$ is Borel measurable. Finally, given $i \in \{ 1,2 \}$ and $m \in \{ 1 , \dotsc , k \}$, let us show that $\psi_m^i$ is Borel measurable. Since the set $\pi_X(B_i)$ is Borel and $\psi_m^i (K) = \emptyset$ for every $K \in X \setminus \pi_X(B_i)$, it suffices to prove that the restriction of $\psi_m^i$ to $\pi_X(B_i)$ is Borel measurable. Clearly, for every $K \in \pi_X(B_i)$ we have
\[ \psi_m^i (K) = P_m^i (K) \cup \big\lbrace \big( f_m^i (K) , 0 \big) \big\rbrace = \bigcup_{j=1}^{i+1} \big\lbrace p_j^i (K,m,t) \, ; \ 0 \leq t \leq t_m^i (K) \big\rbrace . \]
Hence, thanks to Lemma \ref{UnionIsCts}, it is enough to show that the mapping from $\pi_X(B_i)$ to $\mathcal{K} \big( [0,1]^{n+1} \big)$ given by $K \mapsto \big\lbrace p_j^i (K,m,t) \, ; \ 0 \leq t \leq t_m^i (K) \big\rbrace$ is Borel measurable for every $j \in \{ 1 , \dotsc , i+1 \}$. Let $j \in \N$ with $j \leq i+1$ be arbitrary and denote by $T$ the mapping $K \mapsto \big\lbrace p_j^i (K,m,t) \, ; \ 0 \leq t \leq t_m^i (K) \big\rbrace$, $K \in \pi_X(B_i)$. Then we have
\[ T(K) = \Big\lbrace \big( t p_j + (1-t) f_m^i (K) , \, t \big) \, ; \ 0 \leq t \leq t_m^i (K) \Big\rbrace \, , \ \ \ K \in \pi_X(B_i) . \]
Clearly, the mapping from $[0,1]^n \times [0,1]$ to $\mathcal{K} \big( [0,1]^{n+1} \big)$ given by
\[ (x, s) \mapsto \Big\lbrace \big( t p_j + (1-t) x , \, t \big) \, ; \ 0 \leq t \leq s \Big\rbrace \]
is continuous. Moreover, $f_m^i$ is Borel measurable and it is straightforward to verify that the mapping from $\pi_X(B_i)$ to $[0,1]$ given by $K \mapsto t_m^i (K)$ is Borel measurable as well. Therefore, $T$ is Borel measurable and the proof is finished.\claimend
\end{claimproof}
\end{proof}

In order to prove that for every natural number $n$ the homeomorphism ER of compacta in $[0,1]^n$ is Borel reducible to the homeomorphism ER of $n$-dimensional continua in $[0,1]^{n+1}$, we will need a few lemmata.

\begin{lemma}\label{WLOG dim=n}
Let $n \in \N$. The homeomorphism equivalence relation of compacta in $[0,1]^n$ is continuously reducible to the homeomorphism equivalence relation of compacta with nonempty interior in $(0,1)^n$.
\end{lemma}
\begin{proof}
Denote by $X$ the space of all the compact subsets of $(0,1)^n$ which have nonempty interior. It is easy to show that $X$ is an $F_{\sigma}$ set in $\mathcal{K} \big( (0,1)^n \big)$. In particular, as a Borel subset of a Polish space, $X$ forms a standard Borel space. Let $h \colon \R^n \to (0,1)^n$ be an onto homeomorphism. Apparently, the set $(0,1)^n \setminus h \big( [0,1]^n \big)$ is nonempty and open. Choose and fix an open connected set $U \subseteq \R^n$ such that
\[ \emptyset \neq U \subseteq \overline{U} \subseteq (0,1)^n \setminus h \big( [0,1]^n \big). \label{inclusions} \tag{$\star$} \]
Clearly, $\overline{U} \in X$. It follows that the mapping $\Phi \colon \mathcal{K} \big( [0,1]^n \big) \to X$ given by $\Phi (K) = h(K) \cup \overline{U}$ is well-defined. Moreover, since it is easy to see that the mapping from $\mathcal{K} \big( [0,1]^n \big)$ to $\mathcal{K} \big( (0,1)^n \big)$ given by $K \mapsto h(K)$ is continuous, it follows from Lemma \ref{UnionIsCts} that $\Phi$ is continuous as well. Let us show that $\Phi$ is the desired reduction. Let $K,L \in \mathcal{K} \big( [0,1]^n \big)$ be homeomorphic and let $f \colon K \to L$ be an onto homeomorphism. It follows from \eqref{inclusions} that both of the sets $h(K)$ and $h(L)$ are disjoint from $\overline{U}$. This allows us to define a bijection $g \colon \Phi (K) \to \Phi (L)$ by
\begin{equation*}
    g(x) = \begin{cases}
    h \big( f \big( h^{-1}(x) \big) \big) & \textup{if } x \in h(K) \\
    x & \textup{if } x \in \overline{U}.
    \end{cases}
\end{equation*}
Clearly, $g$ is continuous. Hence, as $g$ is a bijection and $\Phi (K)$ is compact, we conclude that $g$ is a homeomorphism between $\Phi (K)$ and $\Phi (L)$.

Conversely, given $K,L \in \mathcal{K} \big( [0,1]^n \big)$ with $\Phi (K)$ homeomorphic to $\Phi (L)$, let us prove that $K$ is homeomorphic to $L$. Let $g \colon \Phi (K) \to \Phi (L)$ be an onto homeomorphism. By \eqref{inclusions} and the definition of $\Phi$, the set $\overline{U}$ is relatively clopen in both $\Phi (K)$ and $\Phi (L)$. Moreover, since $U$ is nonempty and connected, the set $\overline{U}$ is a connected component of both $\Phi (K)$ and $\Phi (L)$. As $g$ is a homeomorphism, it follows that $g \big( \overline{U} \big)$ is a connected component of $\Phi (L)$ and it is relatively clopen in $\Phi (L)$. Similarly, the set $g^{-1} \big( \overline{U} \big)$ is a connected component of $\Phi (K)$ and it is relatively clopen in $\Phi (K)$. As connected components of $\Phi (L)$, the sets $\overline{U}$ and $g \big( \overline{U} \big)$ are either equal or disjoint. Hence, the mapping $\widetilde{g} \colon h(K) \to h(L)$ given by
\begin{equation*}
     \widetilde{g} (x) = \begin{cases}
    g(x) & \textup{if } x \in h(K) \setminus g^{-1} \big( \overline{U} \big) \\
    g \big( g(x) \big) & \textup{if } x \in h(K) \cap g^{-1} \big( \overline{U} \big)
    \end{cases}
\end{equation*}
is a well-defined bijection. Let $A:=h(K) \setminus g^{-1} \big( \overline{U} \big)$, $B:=h(K) \cap g^{-1} \big( \overline{U} \big)$. Since $g^{-1} \big( \overline{U} \big)$ is relatively clopen in $\Phi (K)$, both of the sets $A$ and $B$ are relatively open in $h(K)$. Hence, as both of the mappings $\widetilde{g} \restriction_A$ and $\widetilde{g} \restriction_B$ are (obviously) continuous, it follows that $\widetilde{g}$ is continuous. Therefore, since $h(K)$ is compact and $\widetilde{g}$ is a bijection, we conclude that $\widetilde{g}$ is a homeomorphism between $h(K)$ and $h(L)$. In particular, $h(K)$ and $h(L)$ are homeomorphic, which implies (as $h$ is a homeomorphism) that $K$ is homeomorphic to $L$.
\end{proof}

\begin{lemma}\label{CtblSetToBoundary}
Let $n \in \N$ and denote by $X$ the space of all the compact subsets of $(0,1)^n$ which have nonempty interior. There exist Borel measurable mappings $f_k \colon X \to (0,1)^n$, $k \in \N$, such that:
\begin{enumerate}[label=(\roman*),font=\textup]
    \item $\forall \, K \in X \ \, \forall \, k \in \N : f_k (K) \in (0,1)^n \setminus K ;$
    \item $\forall \, K \in X : \textup{the sequence } \big( f_k(K) \big)_{k=1}^{\infty} \textup{ is injective} ;$
    \item $\forall \, K \in X : \overline{\big\lbrace f_k(K) \, ; \, k \in \N \big\rbrace} = \big\lbrace f_k(K) \, ; \, k \in \N \big\rbrace \cup \partial K ;$
    \item $\forall \, K \in X \ \, \forall \, l \in \N : \textup{the point } f_l (K) \textup{ is isolated in } \overline{\big\lbrace f_k(K) \, ; \, k \in \N \big\rbrace} .$
\end{enumerate}
\end{lemma}
\begin{proof}
For every $k \in \N$ choose and fix a finite set $M_k \subseteq (0,1)^n$ such that $M_k$ is a $2^{-k}$-net for $(0,1)^n$. For every $K \in X$ and every $k \in \N$ let
\[ S_k (K) := \big\lbrace z \in M_k \setminus K \, ; \ \mathrm{dist} (z,K) < 3 \cdot 2^{-k} \big\rbrace . \]
Moreover, for every $K \in X$ define $\displaystyle S(K) := \bigcup_{k=1}^{\infty} S_k (K)$ .

\begin{claim}\label{ClosureOfS(K)}
For every $K \in X$ we have $\overline{S(K)} = S(K) \cup \partial K$.
\end{claim}
\begin{claimproof}
Let $K \in X$ be arbitrary. Since $S(K)$ is disjoint from $K$, it follows that $\overline{S(K)} \cap \Int (K) = \emptyset$. Hence, in order to prove that $\overline{S(K)} \subseteq S(K) \cup \partial K$ it suffices to show that $\overline{S(K)} \setminus K \subseteq S(K)$. Let $x \in \overline{S(K)} \setminus K$ be given. There is $m \in \N$ such that $3 \cdot 2^{-m} < \mathrm{dist} (x,K)$. Denote
\[ A:= \bigcup_{k=1}^m S_k(K) \, , \ \ \ B:= \bigcup_{k=m}^{\infty} S_k(K). \]
It is clear that $A \cup B = S(K)$ and $\mathrm{dist} (z,K) < 3 \cdot 2^{-m}$ for every $z \in B$. In particular, $\mathrm{dist} (z,K) \leq 3 \cdot 2^{-m} < \mathrm{dist} (x,K)$ for every $z \in \overline{B}$. Hence, $x \notin \overline{B}$. Therefore, as $x \in \overline{S(K)} = \overline{A \cup B} = \overline{A} \cup \overline{B}$, we conclude that $x \in \overline{A}$. However, since $S_k (K) \subseteq M_k$ and $M_k$ is finite for every $k \in \N$, it follows that $A$  is finite and $\overline{A} = A$. Thus, $x \in A \subseteq S(K)$.

Now let us show that $S(K) \cup \partial K \subseteq \overline{S(K)}$. Obviously, we just need to prove that $\partial K \subseteq \overline{S(K)}$. Given $x \in \partial K$ and $\varepsilon > 0$, we are going to show that there exists $z \in S(K)$ such that $\lVert z-x \rVert < \varepsilon$. Let $\delta := \frac{1}{2} \mathrm{min} \lbrace 1 , \varepsilon \rbrace$. Since $x \in \partial K$ and $\overline{K} = K \subseteq (0,1)^n$, there is $y \in (0,1)^n \setminus K$ with $\lVert y-x \rVert < \delta$. Let
\[ m:= \mathrm{min} \big\lbrace k \in \N \, ; \ \mathrm{dist} (y,K) > 2^{-k} \big\rbrace . \]
Since $\mathrm{dist} (y,K) \leq \lVert y-x \rVert < \delta \leq \frac{1}{2}$, it follows that $m > 1$. Hence, by the definition of $m$,
\[ 2^{-m} < \mathrm{dist} (y,K) \leq 2^{-(m-1)} . \tag{$\star$} \]
As $M_m$ is a $2^{-m}$-net for $(0,1)^n$, there exists $z \in M_m$ with $\lVert z-y \rVert < 2^{-m}$. The triangle inequality together with ($\star$) yields
\begin{align*}
    \lVert z-x \rVert &\leq \lVert z-y \rVert + \lVert y-x \rVert < 2^{-m} + \delta \\
    &< \mathrm{dist} (y,K) + \delta \leq \lVert y-x \rVert + \delta < 2 \delta \leq \varepsilon.
\end{align*}
Moreover, again by the triangle inequality and ($\star$),
\[ \mathrm{dist} (z,K) \geq \mathrm{dist} (y,K) - \lVert z-y \rVert > 2^{-m} - 2^{-m} =0 , \]
which implies that $z \notin K$. Hence, $z \in M_m \setminus K$. Finally, applying the triangle inequality and ($\star$) once again, we obtain
\[ \mathrm{dist} (z,K) \leq \mathrm{dist} (y,K) + \lVert z-y \rVert < 2^{-(m-1)} + 2^{-m} = 3 \cdot 2^{-m}, \]
which shows that $z \in S_m (K)$. In particular, $z \in S(K)$.\claimend
\end{claimproof}
Denote by $M$ the union of the sets $M_k$, $k \in \N$. It is clear that $M$ is countably infinite. Let $\mu \colon \N \to M$ be an arbitrary bijection. For every $j \in \N$ and every $K \in X$ let $A_j(K) := \big\lbrace i \in \N \, ; \ i \leq j , \, \mu(i) \in S(K)  \big\rbrace$. Moreover, since it is clear that $S(K)$ is an infinite subset of $M$ for every $K \in X$, the mappings ${\alpha}_k \colon X \to \N$, $k \in \N$, given by
\[ {\alpha}_k (K) = \mathrm{min} \big\lbrace j \in \N \, ; \ | A_j (K) | = k \big\rbrace \]
are well-defined. Finally, for every $k \in \N$ define a mapping $f_k \colon X \to (0,1)^n$ by $f_k (K) = \mu \big( {\alpha}_k (K) \big)$. It is easy to see that
\[ \big\lbrace {\alpha}_k (K) \, ; \ k \in \N \big\rbrace = \big\lbrace m \in \N \, ; \ \mu (m) \in S(K) \big\rbrace \]
for every $K \in X$. Moreover, since $\mu$ is a bijection and $S(K) \subseteq M$, we have $\mu \big( \big\lbrace {\alpha}_k (K) \, ; \ k \in \N \big\rbrace \big) = S(K)$ for every $K \in X$. Therefore,
\[ \big\lbrace f_k (K) \, ; \ k \in \N \big\rbrace = \big\lbrace \mu \big( {\alpha}_k (K) \big) \, ; \ k \in \N \big\rbrace = S(K) \]
for every $K \in X$. In particular, the condition (i) is satisfied and, by Claim \ref{ClosureOfS(K)}, so is (iii). Also, since it is clear that the sequence $\big( {\alpha}_k (K) \big)_{k=1}^{\infty}$ is injective for every $K \in X$, it follows from the injectivity of $\mu$ that (ii) is satisfied as well. Moreover, it easily follows from the finiteness of the sets $M_k$, $k \in \N$, and from the definition of the sets $S_k (K)$, $k \in \N$, $K \in X$, that for every $K \in X$ and every $x \in S(K)$ the point $x$ is isolated in $\overline{S(K)}$. Hence, the assertion (iv) is satisfied. It remains to prove that each of the mappings $f_k$, $k \in \N$, is Borel measurable. Equip $\N$ with the discrete topology and $M$ with the subspace topology inherited from $(0,1)^n$. Then (trivially) the mapping $\mu$ is continuous. Hence, it suffices to show that each of the mappings ${\alpha}_k$, $k \in \N$, is Borel measurable. Let $k \in \N$ be arbitrary. Since $\N$ is a countable discrete space, we just need to prove that $\big\lbrace K \in X \, ; \ {\alpha}_k (K) = m \big\rbrace$ is a Borel subset of $X$ for every $m \in \N$. Let $m \in \N$ be given and denote $\Gamma := \big\lbrace P \subseteq \{ 1 , \dotsc , m \} \, ; \ |P|=k \, , \ m \in P \big\rbrace$. Clearly, $\Gamma$ is finite (and possibly empty). It is straightforward to verify that
\[ \big\lbrace K \in X \, ; \ {\alpha}_k (K) = m \big\rbrace = \bigcup_{P \in \Gamma} \, \bigcap_{i=1}^m \big\lbrace K \in X \, ; \ \mu (i) \in S(K) \iff i \in P \big\rbrace . \]
Therefore, we are done once we prove the following claim.
\begin{claim}
The set $\big\lbrace K \in X \, ; \ x \in S(K) \big\rbrace$ is Borel in $X$ for every $x \in (0,1)^n$.
\end{claim}
\begin{claimproof}
Let $x \in (0,1)^n$ be given. By the definition of $S(K)$, we have
\[ \big\lbrace K \in X \, ; \ x \in S(K) \big\rbrace = \bigcup_{l=1}^{\infty} \big\lbrace K \in X \, ; \ x \in S_l (K) \big\rbrace . \]
Hence, it suffices to show that each of the sets $\big\lbrace K \in X \, ; \ x \in S_l (K) \big\rbrace$, $l \in \N$, is Borel in $X$. Let $l \in \N$ be arbitrary. By the definition of $S_l (K)$, the set $\big\lbrace K \in X \, ; \ x \in S_l (K) \big\rbrace$ is equal to
\[ \big\lbrace K \in X \, ; \ x \in M_l \big\rbrace \cap \big\lbrace K \in X \, ; \ x \notin K \big\rbrace \cap \big\lbrace K \in X \, ; \ \mathrm{dist} (x,K) < 3 \cdot 2^{-l} \big\rbrace . \]
Clearly, the first of these three sets is clopen in $X$, the second one is $F_{\sigma}$ in $X$ and the third one is open in $X$. In particular, each of them is Borel. \claimend
\end{claimproof}
\end{proof}

\begin{theorem}\label{thmCompCont}
The homeomorphism equivalence relation of compacta in $[0,1]^n$ is Borel reducible to the homeomorphism equivalence relation of $n$-dimensional continua in $[0,1]^{n+1}$ for every $n\in\N$. In particular, we have $~{H_n\leq_B C_{n+1}}$ for every $n\in\N$.
\end{theorem}

\begin{proof}
Let $n \in \N$ be given. By \cite[Theorem 4.2]{ChangGaoDim}, the equivalence relation $H_1$ is Borel bireducible with the isomorphism equivalence relation of countable graphs. Thus, repeating the arguments presented at the beginning of the proof of Theorem \ref{AmbientHomeos}, we conclude that $H_1$ is Borel reducible to the homeomorphism equivalence relation of $1$-dimensional continua in $[0,1]^2$. Hence, let us assume that $n > 1$. Let $X$ be the space of all the compact subsets of $(0,1)^n$ which have nonempty interior and denote $Y := \mathcal{K} \big( [0,1]^n \times [-1,1] \big)$. Clearly, the homeomorphism equivalence relation of $n$-dimensional continua in $[0,1]^{n+1}$ is continuously bireducible with the homeomorphism equivalence relation of $n$-dimensional continua in $[0,1]^n \times [-1,1]$. Using this fact together with Lemma \ref{WLOG dim=n}, we immediately see that it suffices to find a Borel measurable mapping $\Phi \colon X \to Y$ such that $\Phi$ is a reduction from the homeomorphism equivalence relation on $X$ to the homeomorphism equivalence relation on $Y$ and $\Phi(K)$ is $n$-dimensional and connected for every $K \in X$.

Let $B := \big\lbrace (K,y) \in X \times (0,1)^n \, ; \ y \in \Int (K) \big\rbrace$. It easily follows from Lemma \ref{BorelSets} that $B$ is Borel in $X \times (0,1)^n$. Moreover, for every $K \in X$ we have $\big\lbrace y \in (0,1)^n \, ; \ (K,y) \in B \big\rbrace = \Int (K)$. In particular, $B$ is a Borel set whose nonempty vertical sections are infinite and $\sigma$-compact. In addition, it immediately follows from the definitions of $X$ and $B$ that $\pi_X(B) = X$. Thus, by Proposition \ref{uniformization}, there exist Borel measurable mappings $f_k^1 \colon X \to (0,1)^n$, $k \in \N$, such that:
\begin{enumerate}[label=(\roman*),font=\textup]
    \item $\forall \, K\in X : \textup{the set } \big\lbrace f_k^1 (K) \, ; \, k\in\N \big\rbrace \textup{ is a dense subset of } \Int (K) ;$
    \item $\forall \, K\in X : \textup{the sequence } \big( f_k^1 (K) \big)_{k=1}^{\infty} \textup{ is injective} .$
\end{enumerate}
By Lemma \ref{CtblSetToBoundary}, there exist Borel measurable mappings $f_k^2 \colon X \to (0,1)^n$, $k \in \N$, such that:
\begin{enumerate}[label=(\arabic*),font=\textup]
    \item $\forall \, K \in X \ \, \forall \, k \in \N : f_k^2 (K) \in (0,1)^n \setminus K ;$
    \item $\forall \, K \in X : \textup{the sequence } \big( f_k^2 (K) \big)_{k=1}^{\infty} \textup{ is injective} \, ;$
    \item $\forall \, K \in X : \overline{\big\lbrace f_k^2 (K) \, ; \, k \in \N \big\rbrace} = \big\lbrace f_k^2 (K) \, ; \, k \in \N \big\rbrace \cup \partial K ;$
    \item $\forall \, K \in X \ \, \forall \, l \in \N : \textup{the point } f_l^2 (K) \textup{ is isolated in } \overline{\big\lbrace f_k^2 (K) \, ; \, k \in \N \big\rbrace} .$
\end{enumerate}
Fix an arbitrary point $p_0 \in [0,1]^n \times \{ 1 \}$ and for every $x \in (0,1)^n \times \{ 0 \}$ let
\[ P_0 (x) := \big\lbrace tx+(1-t)p_0 \, ; \ t \in [0,1] \big\rbrace . \]
For every $K \in X$ denote
\begin{align*}
    F_1 (K) &:= \bigcup \big\lbrace P_0 (x) \, ; \ x \in \partial K \times \{ 0 \} \big\rbrace , \\
    F_2 (K) &:= \bigcup \Big\lbrace P_0 (x) \, ; \ x \in \big\lbrace f_k^2 (K) \, ; \, k\in\N \big\rbrace \times \{ 0 \} \Big\rbrace , \\
    F(K) &:= F_1 (K) \cup F_2 (K).
\end{align*}
\begin{claim}\label{F(K)inY}
The set $F(K)$ is compact for every $K \in X$.
\end{claim}
\begin{claimproof}
Let $K \in X$ be given. By (3) and the definition of $F(K)$, we have
\[ F(K) = \bigcup \Big\lbrace P_0 (x) \, ; \ x \in \overline{\big\lbrace f_k^2 (K) \, ; \, k\in\N \big\rbrace} \times \{ 0 \} \Big\rbrace . \]
Hence, denoting $Z := \overline{\big\lbrace f_k^2 (K) \, ; \, k\in\N \big\rbrace} \times \{ 0 \}$, it follows that $F(K)$ is the image of $Z \times [0,1]$ under the mapping
\[ (x,t) \mapsto tx+(1-t)p_0 \, , \ \ (x,t) \in \R^{n+1} \times [0,1] . \]
Therefore, since this mapping is continuous and the set $Z \times [0,1]$ is compact, we are done.\claimend
\end{claimproof}
Choose and fix three distinct points $p_1 , p_2 , p_3 \in [0,1]^n \times \{ -1 \}$. Similarly as in the proof of Theorem \ref{AmbientHomeos}, it is possible to construct Borel measurable mappings $t_m^i \colon X \to (0,1)$, $i \in \{ 1,2 \}$, $m \in \N$, such that $\displaystyle \lim_{m \to \infty} t_m^i (K) = 0$ for all $i \in \{ 1,2 \}$ and $K \in X$ and the sets
\begin{align*}
    &P_m^i (K) := \Big\lbrace tp_j + (1-t) \big( f_m^i (K) , 0 \big) \, ; \ j \leq i+1 \, , \ 0 < t \leq t_m^i (K) \Big\rbrace \, ,\\
    &i \in \{ 1,2 \} \, , \ m \in \N \, , \ K \in X \, ,
\end{align*}
satisfy:
\begin{enumerate}[label=(\alph*),font=\textup]
    \item $\forall \, K \in X \ \, \forall \, i \in \{ 1,2 \} \ \, \forall \, k , l \in \N : \ k \neq l \implies P_k^i (K) \cap P_l^i (K) = \emptyset \, ;$
    \item $\forall \, K \in X \ \, \forall \, k \in \N : \, P_k^1 (K) \subseteq \Int (K) \times [-1,0) \, ;$
    \item $\forall \, K \in X \ \, \forall \, k \in \N : \, P_k^2 (K) \subseteq \big( [0,1]^n \setminus K \big) \times [-1,0) .$
\end{enumerate}
For every $i \in \{ 1,2 \}$ and every $K \in X$ let $\displaystyle P^i (K) := \bigcup_{m=1}^{\infty} P_m^i (K)$. 

\begin{claim}\label{BelongsToY}
The set $\big( K \times \{ 0 \} \big) \cup F(K) \cup P^1 (K) \cup P^2 (K)$ belongs to $Y$ for every $K \in X$.
\end{claim}
\begin{claimproof}
Let $K \in X$ be given and denote
\[ Z:= \big( K \times \{ 0 \} \big) \cup F(K) \cup P^1 (K) \cup P^2 (K) . \]
It is clear that $Z \subseteq [0,1]^n \times [-1,1]$. Therefore, it suffices to show that $Z$ is closed. To that end, we shall use Lemma \ref{ClosednessOfUnion}. First, thanks to Claim \ref{F(K)inY}, the set $\big( K \times \{ 0 \} \big) \cup F(K)$ is closed. Clearly, $\overline{P_m^i (K)} = P_m^i (K) \cup \big\lbrace \big( f_m^i (K) , 0 \big) \big\rbrace$
for all $i \in \{ 1,2 \}$ and $m \in \N$. Thus, as $f_m^1 (K) \in K$ and $\big( f_m^2 (K) , 0 \big) \in F(K)$ for every $m \in \N$, we indeed have $\overline{P_m^i (K)} \setminus P_m^i (K) \subseteq \big( K \times \{ 0 \} \big) \cup F(K)$ and $\overline{P_m^i (K)} \cap \big( \big( K \times \{ 0 \} \big) \cup F(K) \big) \neq \emptyset$ for every $i \in \{ 1,2 \}$ and $m \in \N$. Finally, for each $i \in \{ 1,2 \}$, it is straightforward to verify that $\mathrm{diam} \big( P_m^i (K) \big) \rightarrow 0$. Hence, by Lemma \ref{ClosednessOfUnion}, the set $Z$ is closed. \claimend
\end{claimproof}
Define the desired mapping $\Phi \colon X \to Y$ by
\[ \Phi (K) = \big( K \times \{ 0 \} \big) \cup F(K) \cup P^1 (K) \cup P^2 (K) . \]
\begin{claim}
The set $\Phi (K)$ is $n$-dimensional and connected for every $K \in X$.
\end{claim}
\begin{claimproof}
Let $K \in X$ be given. The connectedness (even path-connectedness) of $\Phi (K)$ is straightforward, so let us focus on the $n$-dimensionality. By the definition of $X$, it is clear that $K$ is $n$-dimensional. Hence, so is $K \times \{ 0 \}$ and the dimension of $\Phi (K)$ is therefore at least $n$. On the other hand, as the interior of $\Phi (K)$ with respect to $\R^{n+1}$ is empty (since $\Phi (K)$ can easily be shown to be meagre in $\R^{n+1}$), it follows from \cite[Corollary 1.8.4, Theorem 1.8.11]{EngelkingDimension} that the dimension of $\Phi (K)$ is equal to $n$.\claimend
\end{claimproof}
\begin{claim}
The mapping $\Phi$ is a reduction from the homeomorphism equivalence relation on $X$ to the homeomorphism equivalence relation on $Y$.
\end{claim}
\begin{claimproof}
Let $K,L \in X$ be homeomorphic and let $h \colon K \to L$ be an onto homeomorphism. By the domain invariance theorem, $h \big( \Int (K) \big) = \Int (L)$ and $h(\partial K) = \partial L$. Hence, thanks to (1), (3) and (4), it follows from \cite[Proposition 2]{KrupskiVejnar} that there exists an onto homeomorphism
\[ h_0 \colon \big\lbrace f_k^2 (K) \, ; \, k \in \N \big\rbrace \cup \partial K \to \big\lbrace f_k^2 (L) \, ; \, k \in \N \big\rbrace \cup \partial L \]
with $h_0 \big( \big\lbrace f_k^2 (K) \, ; \, k \in \N \big\rbrace \big) = \big\lbrace f_k^2 (L) \, ; \, k \in \N \big\rbrace$ and $h_0 (x) = h(x)$, $x \in \partial K$. Define a mapping
$g \colon K \cup \big\lbrace f_k^2 (K) \, ; \, k \in \N \big\rbrace \to L \cup \big\lbrace f_k^2 (L) \, ; \, k \in \N \big\rbrace$
by
\begin{equation*}
    g(x)= \begin{cases}
    h(x) & \textup{if } x \in K ;\\
    h_0(x) & \textup{if } x \in \big\lbrace f_k^2 (K) \, ; \, k \in \N \big\rbrace .\end{cases}
\end{equation*}
Clearly, $g$ is an onto homeomorphism. Denote $D_0 := \big\lbrace f_k^1 (K) \, ; \, k\in\N \big\rbrace$, $D_2 := \big\lbrace f_k^1 (L) \, ; \, k\in\N \big\rbrace$ and $D_1 := g(D_0) = h(D_0)$. Both of the sets $D_1$, $D_2$ are countable dense subsets of $\Int (L)$. Therefore, by Lemma \ref{SCDH}, there exists an onto homeomorphism $f \colon \R^n \to \R^n$ such that $f(D_1) = D_2$ and $f(x) = x$ for every $x \in \R^n \setminus \Int (L)$. It follows that $f \circ g$ is a homeomorphism between $K \cup \big\lbrace f_k^2 (K) \, ; \, k \in \N \big\rbrace$ and $L \cup \big\lbrace f_k^2 (L) \, ; \, k \in \N \big\rbrace$ and it satisfies:
\begin{itemize}
   \item $(f \circ g) \big( \big\lbrace f_k^i (K) \, ; \, k \in \N \big\rbrace \big) = \big\lbrace f_k^i (L) \, ; \, k \in \N \big\rbrace , \ i \in \{ 1,2 \} ;$
   \item $(f \circ g) (\partial K) = \partial L \, ;$
   \item $(f \circ g) \big( \Int (K) \big) = \Int (L) .$
\end{itemize}
From now on, we can proceed similarly as in the corresponding part of the proof of Theorem \ref{AmbientHomeos} (see the proof of Claim \ref{xxx}) and we come to the conclusion that $\Phi (K)$ is homeomorphic to $\Phi (L)$.

Conversely, assume that we are given $K,L \in X$ with $\Phi (K)$ homeomorphic to $\Phi (L)$. Let $h \colon \Phi (K) \to \Phi (L)$ be an onto homeomorphism. Similarly as in the corresponding part of the proof of Theorem \ref{AmbientHomeos} (see the proof of Claim \ref{yyy}) we deduce that $h \big( \big\lbrace f_k^i (K) \, ; \, k\in\N \big\rbrace \times \{ 0 \} \big) = \big\lbrace f_k^i (L) \, ; \, k\in\N \big\rbrace \times \{ 0 \}$ for every $i \in \{ 1,2 \}$. Hence, thanks to (i), (1) and (3), we have
\[ h \big( \overline{\Int (K)} \times \{ 0 \} \big) = \overline{\Int (L)} \times \{ 0 \} \, , \ \ h \big( \partial K \times \{ 0 \} \big) = \partial L \times \{ 0 \} . \]
Thus, $h \big( K \times \{ 0 \} \big) = L \times \{ 0 \}$. In particular, $K$ is homeomorphic to $L$.\claimend
\end{claimproof}
Finally, proceeding similarly as in the corresponding part of the proof of Theorem \ref{AmbientHomeos} (see the proof of Claim \ref{zzz}), it can be verified that $\Phi$ is Borel measurable.
\end{proof}

\section{Non-classification results}
In this section we prove that neither the homeomorphism ER of absolute retracts in $\R^3$ nor the homeomorphism ER of $1$-dimensional locally connected continua in the plane is classifiable by countable structures.

Recall that if $X$ is a metric space, $Y\subseteq X$ and $\varepsilon>0$ we say that $X$ is $\varepsilon$-deformable into $Y$ if there exists a continuous mapping $\varphi\colon X\times [0,1]\to X$ such that $\varphi(x,0)=x$, $\varphi(x,1)\in Y$ and the diameter of $\varphi(\{x\}\times[0,1])$ is at most $\varepsilon$ for every $x\in X$.
The following lemma was proved in \cite[1.1 and 1.3]{Krasinkiewicz}.

\begin{lemma}\label{AR}
Let $X$ be a compact metric space such that for every $\varepsilon>0$ there exists an absolute retract $Y\subseteq X$ for which $X$ is $\varepsilon$-deformable into $Y$. Then $X$ is an absolute retract.
\end{lemma}

\begin{theorem}\label{thmAR3}
The homeomorphism equivalence relation of absolute retracts in $\R^3$ is not classifiable by countable structures.
\end{theorem}

\begin{proof}
Consider the Hilbert cube ${[1,2]}^{\N}$ and the equivalence relation $E$ on ${[1,2]}^{\N}$ given by $\, x E y \iff \lim\limits_{n\to \infty} \left| x_n-y_n \right|=0$. (The interval $[1,2]$ was chosen instead of the usual $[0,1]$ for technical reasons only.) It is known that $E$ is not classifiable by countable structures; see \cite[Lemma 17]{KrupskiVejnar} for the proof. Hence, it suffices to show that $E$ is Borel reducible to the homeomorphism ER on $\mathsf{AR} (\R^3)$.

Denote $P:= \big[0 , \frac{1}{2} \big] \times [0,3] \times \lbrace 0 \rbrace$. To each $x \in {[1 , 2 ]}^{\N}$ we are going to assign an absolute retract $\Phi (x) \subseteq \R^3$ consisting of the ``plate" $P$, a fixed sequence of markers connected to $\lbrace 0 \rbrace \times [0,3] \times \{ 0 \}$, which makes this segment rigid, and an $x$-dependent sequence of other markers attached to $\big( 0 , \frac{1}{2} \big] \times [1,2] \times \lbrace 0 \rbrace$ which encode the asymptotic information about the sequence $x$.

For every $n \in \N$ denote ${\mathcal{J}}_n := \lbrace k \in \mathbb{Z} \, ; -n \leq k \leq n \rbrace$ and ${\mathcal{I}}_n := {\mathcal{J}}_n \setminus \lbrace 0 \rbrace$.\\
For every $n \in \N$ and $k \in {\mathcal{I}}_n$ define a mapping $f_k^n \colon \R^2 \to \R^3$ by
\[f_k^n (z,t) = \left( 2^{-n} , \, z+\frac{tk}{n} , \, t \right) .\]
Let ${(q_i)}_{i=1}^{\infty}$ be an injective sequence in $[0,3]$ such that $\overline{\lbrace q_i \, ; \, i \in \N \rbrace} = [0,3]$. For every $n \in \N$ and $k \in {\mathcal{J}}_n$ define a mapping $g_k^n \colon \R \to \R^3$ by
\[g_k^n (t) = \left(-t , \, q_n , \, \frac{tk}{n} \right) .\]
Finally, define a mapping $\Phi \colon {[1,2]}^{\N} \to \mathsf{AR} (\R^3)$ by
\[\Phi (x) = P \cup \left( \bigcup_{n=1}^{\infty} \bigcup_{k \in {\mathcal{J}}_n} g_k^n \Big( \big[0 , 2^{-n} \big] \Big) \right) \cup \left( \bigcup_{n=1}^{\infty} \bigcup_{k \in {\mathcal{I}}_n} f_k^n \Big( \{ x_n \} \times \big[0 , 2^{-n} \big] \Big) \right).\]
The following claim shows that $\Phi$ is well-defined.

\begin{claim}
The set $\Phi(x)$ is an absolute retract for every $x \in {[1,2]}^{\N}$.
\end{claim}

\begin{claimproof}
Let $x \in {[1,2]}^{\N}$ be given. It is straightforward to verify (e.g., using Lemma \ref{ClosednessOfUnion}) that $\Phi(x)$ is a compact set. To show that $\Phi(x)$ is an absolute retract, we shall use Lemma~\ref{AR}. Given $\varepsilon > 0$, we need to verify that there exists an absolute retract $Y \subseteq \Phi (x)$ such that $\Phi (x)$ is $\varepsilon$-deformable into $Y$. Clearly, there is $n_0 \in \N$ such that
\[\mathrm{diam} \Big( g_k^n \big( \big[0 , 2^{-n} \big] \big) \Big) < \varepsilon \, , \ \ \ n > n_0 \, , \ k \in {\mathcal{J}}_n \, , \]
and
\[\mathrm{diam} \Big( f_k^n \Big( \{ x_n \} \times \big[ 0 , 2^{-n} \big] \Big) \Big) < \varepsilon \, , \ \ \ n > n_0 \, , \ k \in {\mathcal{I}}_n \, .\]
Define
\[Y:=P \cup \left( \bigcup_{n=1}^{n_0} \bigcup_{k \in {\mathcal{J}}_n} g_k^n \Big( \big[0 , 2^{-n} \big] \Big) \right) \cup \left( \bigcup_{n=1}^{n_0} \bigcup_{k \in {\mathcal{I}}_n} f_k^n \Big( \{ x_n \} \times \big[0 , 2^{-n} \big] \Big) \right).\]
Then $Y \subseteq \Phi (x)$ and, by \cite[Theorem 1.2.16]{vanMill}, $Y$ is an absolute retract.\\
Define a mapping $\varphi \colon \Phi (x) \times [0,1] \to \Phi (x)$ by
\begin{align*}
    \varphi (p,s) &= p \, ,&&(p,s) \in Y \times [0,1] ,\\
    \varphi \big( g_k^n (t) , s \big) &= g_k^n (t-st) \, ,&&n > n_0 \, , \ k \in {\mathcal{J}}_n \, , \ t \in \big( 0 , 2^{-n} \big] , \ s \in [0,1] ,\\
    \varphi \big( f_k^n ( x_n , t ) , s \big) &= f_k^n ( x_n , t-st ) \, ,&&n > n_0 \, , \ k \in {\mathcal{I}}_n \, , \ t \in \big( 0 , 2^{-n} \big] , \ s \in [0,1] .
\end{align*}
It is straightforward to show that $\varphi$ is continuous, $\varphi \big( \Phi (x) \times \{ 1 \} \big) = Y$, $\varphi (p,0) = p$ for every $p \in \Phi (x)$ and $\mathrm{diam} \big( \varphi \big( \{ p \} \times [0,1] \big) \big) < \varepsilon$ for every $p \in \Phi (x)$. Hence, $\Phi (x)$ is $\varepsilon$-deformable into $Y$.
\claimend
\end{claimproof}

Let us show that $\Phi$ is a reduction from $E$ to the homeomorphism ER on $\mathsf{AR} ( \R^3 )$. Let $x , y \in {[1,2]}^{\N}$ be given.

\begin{claim}
If $xEy$ then $\Phi(x)$ is homeomorphic to $\Phi(y)$.
\end{claim}

\begin{claimproof}
Assume that $x E y$. Define a function $h \colon [1,2] \times [1,2] \times [0,3] \to [0,3]$ by
\begin{equation*}
    h (a,b,t) = \begin{cases}
    \dfrac{tb}{a} & \textup{if } t \leq a\\
    \\
    \dfrac{3b - 3a + t (3-b)}{3-a} & \textup{if } t > a.
\end{cases} \ \ 
\end{equation*}
Then $h$ is continuous and for all $a , b \in [1,2]$ the function $t \mapsto h (a,b,t)$ is a self-homeomorphism of $[0,3]$ and it maps $a$ to $b$.\\
Define a mapping $\varphi \colon \Phi (x) \to \Phi (y)$
by the following formulae:
\begin{gather*}
    \varphi (p) = p \ , \ \ \ \ p \in \big( \{ 0 \} \times [0,3] \times \{ 0 \} \big) \cup \bigcup_{n=1}^{\infty} \bigcup_{k \in {\mathcal{J}}_n} g_k^n \Big( \big[ 0 , 2^{-n} \big] \Big) \, ,\\
    \varphi \big( f_k^n ( x_n , t ) \big) = f_k^n ( y_n , t ) \, , \ \ \ \ n \in \N \, , \ k \in {\mathcal{I}}_n \, , \ t \in \big( 0 , 2^{-n} \big] \, ,\\
    \varphi ( 2^{-n} , t , 0 ) = \big( 2^{-n} , h ( x_n , y_n , t ) , 0 \big) \, , \ \ \ \ n \in \N \, , \ t \in [0,3] \, ,\\
    \varphi \big( s \cdot 2^{-n-1} + (1-s) \cdot 2^{-n} , t , 0 \big) = s \cdot \varphi ( 2^{-n-1} , t , 0 ) + (1-s) \cdot \varphi ( 2^{-n} , t , 0 ) \, ,\\
    n \in \N \, , \ s \in (0,1) \, , \ t \in [0,3] \, .
\end{gather*}
The mapping $\varphi$ is well-defined and it is easy to verify that it is a bijection. Hence, in order to show that $\varphi$ is a homeomorphism it suffices to prove (since $\Phi (x)$ is compact) that $\varphi$ is continuous.
Continuity of $\varphi$ at points of the set $\Phi (x) \setminus \big( \{ 0 \} \times [0,3] \times \{ 0 \} \big)$ is straightforward and it does not depend on the assumption that $x E y$. Let $z \in [0,3]$ and $\varepsilon > 0$ be given. We would like to find $\delta > 0$ such that $\norm{\varphi (p) - \varphi (0,z,0)} < \varepsilon$ for every $p \in \Phi (x)$ satisfying $\norm{p - (0,z,0)} < \delta$. By the assumption that $x E y$, there is $m \in \N$ such that for every $n \in \N$ with $n \geq m$ we have $|x_n-y_n| < \frac{\varepsilon}{2}$. Also, as $h$ is continuous and its domain is compact, $h$ is uniformly continuous. Hence (since $x E y$) there exists $n_0 \in \N$ such that $\big| h (x_n , y_n , t) - h (x_n , x_n , t) \big| < \frac{\varepsilon}{2}$ for every $t \in [0,3]$ and every $n \in \N$ with $n \geq n_0$. It follows that for every $t \in [0,3]$ and every $n \in \N$ satisfying $n \geq n_0$ we have
\begin{align*}
    \norm{\varphi (2^{-n},t,0) - (2^{-n},t,0)} &= \norm{\big( 2^{-n},h(x_n,y_n,t),0 \big) - (2^{-n},t,0)}\\
    &= \norm{\big( 0,h(x_n,y_n,t)-t,0 \big)} = |h(x_n,y_n,t) - t|\\
    &= \big| h(x_n,y_n,t) - h(x_n,x_n,t) \big| < \tfrac{1}{2} \varepsilon .
\end{align*}
Define $\delta := \mathrm{min} \big\lbrace \tfrac{1}{2} \varepsilon , \, 2^{-m} , \, 2^{- {n}_0} \big\rbrace$ and let $p \in \Phi (x)$ be an arbitrary point such that $\norm{p - (0,z,0)} < \delta$. We will distinguish four cases:

If $p = (2^{-n},t,0)$ for some $n \in \N$ and $t \in [0,3]$, then
\[2^{-n} \leq \norm{(2^{-n},t-z,0)} = \norm{p - (0,z,0)} < \delta \leq 2^{-n_0} ,\]hence $n > n_0$, which gives us $\norm{\varphi (2^{-n},t,0) - (2^{-n},t,0)} < \frac{1}{2} \varepsilon$. Therefore,
\begin{gather*}
    \norm{\varphi (p) - \varphi (0,z,0)} = \norm{\varphi (2^{-n},t,0) - (0,z,0)} < \frac{\varepsilon}{2} + \norm{p - (0,z,0)}\\
    < \frac{\varepsilon}{2} + \delta \leq \frac{\varepsilon}{2} + \frac{\varepsilon}{2} = \varepsilon .
\end{gather*}

If $p = \big( s \cdot 2^{-n-1} + (1-s) \cdot 2^{-n} , \ t , \ 0 \big)$ for some $n \in \N$, $s \in (0,1)$ and $t \in [0,3]$, then
\begin{gather*}
    2^{-n-1} \leq s \cdot 2^{-n-1} + (1-s) 2^{-n} \leq \norm{\big( s \cdot 2^{-n-1} + (1-s) 2^{-n} , t-z , 0 \big) }\\
    = \norm{p - (0,z,0)}  < \delta \leq 2^{-n_0} ,
\end{gather*}
hence $n + 1 > n_0$ and $n \geq n_0$, which yields $\norm{\varphi (2^{-n},t,0) - (2^{-n},t,0)} < \tfrac{1}{2} \varepsilon$ and $\norm{\varphi (2^{-n-1},t,0) - (2^{-n-1},t,0)} < \tfrac{1}{2} \varepsilon$. Therefore,
\begin{gather*}
    \norm{\varphi (p) - \varphi (0,z,0)} = \norm{\varphi \big( s \cdot 2^{-n-1} + (1-s) 2^{-n} , t , 0 \big) - (0,z,0)}\\
    = \norm{s \cdot \varphi (2^{-n-1},t,0) + (1-s) \cdot \varphi (2^{-n},t,0) - (0 , z , 0)}\\
    \leq \norm{s \cdot (2^{-n-1},t,0) + (1-s) \cdot (2^{-n},t,0) - (0,z,0)}\\
    + s \cdot \norm{\varphi (2^{-n-1},t,0) - (2^{-n-1},t,0)} + (1-s) \cdot \norm{\varphi (2^{-n},t,0) - (2^{-n},t,0)}\\
    = \norm{p - (0,z,0)} + s \cdot \norm{\varphi (2^{-n-1},t,0) - (2^{-n-1},t,0)}\\
    + (1-s) \cdot \norm{\varphi (2^{-n},t,0) - (2^{-n},t,0)}\\
    < \delta + s \cdot \frac{\varepsilon}{2} + (1-s) \cdot \frac{\varepsilon}{2} = \delta + \frac{\varepsilon}{2} \leq \frac{\varepsilon}{2} + \frac{\varepsilon}{2} = \varepsilon .
\end{gather*}

If $p = f_k^n (x_n,t)$ for some $n \in \N$, $k \in {\mathcal{I}}_n$ and $t \in \big( 0 , 2^{-n} \big]$, then
\begin{gather*}
    2^{-n} \leq \norm{ \Big( 2^{-n} , x_n + \frac{tk}{n} - z , t \Big) } = \norm{ \Big( 2^{-n} , x_n + \frac{tk}{n} , t \Big) - (0,z,0)}\\
    = \norm{f_k^n (x_n,t) - (0,z,0)} = \norm{p - (0,z,0)} < \delta \leq 2^{-m} ,
\end{gather*}
hence $n > m$, which implies that $|x_n-y_n| < \frac{1}{2} \varepsilon$. Therefore,
\begin{gather*}
    \norm{\varphi (p) - \varphi (0,z,0)} = \norm{\varphi \big( f_k^n (x_n,t) \big) - \varphi (0,z,0)} = \norm{f_k^n (y_n,t) - (0,z,0)}\\
    \leq \norm{f_k^n (y_n,t) - f_k^n (x_n,t)} + \norm{f_k^n (x_n,t) - (0,z,0)}\\
    = \norm{ \Big( 2^{-n} , y_n + \frac{tk}{n} , t \Big) - \Big( 2^{-n} , x_n + \frac{tk}{n} , t \Big) } + \norm{p - (0,z,0)}\\
    = \norm{(0,y_n-x_n,0)} + \norm{p - (0,z,0)} = |x_n-y_n| + \norm{p - (0,z,0)}\\
    < \frac{\varepsilon}{2} + \delta \leq \frac{\varepsilon}{2} + \frac{\varepsilon}{2} = \varepsilon .
\end{gather*}
Finally, if $p \in \displaystyle \big( \{ 0 \} \times [0,3] \times \{ 0 \} \big) \cup \bigcup_{n=1}^{\infty} \bigcup_{k \in {\mathcal{J}}_n} g_k^n \big( \big[ 0 , 2^{-n} \big] \big)$, we immediately receive $\norm{\varphi (p) - \varphi (0,z,0)} = \norm{p - (0,z,0)} < \delta \leq \dfrac{\varepsilon}{2} < \varepsilon$. \claimend
\end{claimproof}

\begin{claim}
If $\Phi (x)$ is homeomorphic to $\Phi (y)$ then  $x E y$. 
\end{claim}

\begin{claimproof}
We proceed by contradiction. Assume that $\Phi (x)$ is homeomorphic to $\Phi (y)$, and yet $(x,y) \notin E$, i.e. the sequence $x - y$ does not converge to 0. Let $\varphi \colon \Phi (x) \to \Phi (y)$ be an onto homeomorphism. Observe that for every $w \in {[1,2]}^{\N}$ and every $n \in \N$ the point $(0,q_n,0)$ is the unique point $p \in \Phi (w)$ such that the set $\Phi (w) \setminus \{ p \}$ has exactly $2n+2$ connected components – the sets $g_k^n \big( \big( 0 , 2^{-n} \big] \big)$, $k \in {\mathcal{J}}_n$, and the complement of their union in $\Phi (w) \setminus \{ p \}$. Similarly, for every $w \in {[1,2]}^{\N}$ and every $n \in \N$ the point $(2^{-n},w_n,0)$ is the unique point $p \in \Phi (w)$ such that the set $\Phi (w) \setminus \{ p \}$ has exactly $2n+1$ connected components – the sets $f_k^n \big( \{ w_n \} \times \big( 0 , 2^{-n} \big] \big)$, $k \in {\mathcal{I}}_n$, and the complement of their union in $\Phi (w) \setminus \{ p \}$. Using this observation together with the fact that $\varphi$ is a homeomorphism and the set $\lbrace q_n ; n \in \N \rbrace$ is dense in $[0,3]$, we conclude that $\varphi (p) = p$ for every $p \in \{ 0 \} \times [0,3] \times \{ 0 \}$ and $\varphi (2^{-n},x_n,0) = (2^{-n},y_n,0)$ for every $n \in \N$.

Since $x$ and $y$ are sequences in $[1,2]$ and the sequence $x - y$ does not converge to 0, standard compactness arguments show that we can find two distinct numbers $c_x , c_y \in [1,2]$ and an increasing sequence ${(n_k)_{k=1}^{\infty}}$ of natural numbers such that
\[\lim\limits_{k\to \infty} x_{n_k} = c_x \neq c_y = \lim\limits_{k\to \infty} y_{n_k} .\]
By the continuity of $\varphi$, we have
\[\lim\limits_{k\to \infty} \varphi \big( 2^{-{n_k}},x_{n_k},0 \big) = \varphi (0,c_x,0) = (0,c_x,0) .\]
On the other hand,
\[\lim\limits_{k\to \infty} \varphi \big( 2^{-{n_k}} , x_{n_k} , 0 \big) = \lim\limits_{k\to \infty} \big( 2^{-{n_k}} , y_{n_k} , 0 \big) = (0,c_y,0) ,\]
therefore $c_x = c_y$, which is a contradiction.
\claimend
\end{claimproof}

It remains to verify that $\Phi$ is Borel measurable. But in fact, we prove that $\Phi$ is actually continuous:

\begin{claim}
The mapping $\Phi$ is continuous.
\end{claim}

\begin{claimproof}
Denote by $d_H$ the Hausdorff metric on $\mathcal{K} (\R^3)$ induced by the usual Euclidean metric on $\R^3$. Then $d_H$ is compatible with the topology of $\mathcal{K} (\R^3)$, hence - since we consider $\mathsf{AR} (\R^3)$ with the subspace topology inherited from $\mathcal{K} (\R^3)$ - it suffices to show that $\Phi$ is continuous as a mapping from ${[1,2]}^{\N}$ to the metric space $\big( \mathcal{K} (\R^3) , d_H \big)$.

Let $x \in {[1,2]}^{\N}$ and $\varepsilon > 0$ be given. Clearly, for every $n \in \N$, $k \in {\mathcal{I}}_n$ and $z \in [1,2]$ we have
\begin{gather*}
    \mathrm{diam} \Big( f_k^n \big( \{ z \} \times \big[ 0 , 2^{-n} \big] \big) \Big) = \norm{f_k^n \big( z , 2^{-n} \big) - f_k^n (z,0)} = 2^{-n} \norm{\big( 0 , \tfrac{k}{n} , 1 \big)} \\
    \leq 2^{-n} \norm{(0,1,1)} = \sqrt{2} \cdot 2^{-n}.
\end{gather*}
In particular, there exists $n_0 \in \N$ such that for every $n > n_0$, $k \in {\mathcal{I}}_n$ and $z \in [1,2]$ we have
\[\mathrm{diam} \Big( f_k^n \big( \{ z \} \times \big[ 0 , 2^{-n} \big] \big) \Big) < \varepsilon .\]
Let $U:= \Big\lbrace y \in {[1,2]}^{\N} ; \ \forall \, n \in \N : n \leq n_0 \implies |x_n-y_n| < \varepsilon \Big\rbrace$. Then $U$ is a neighbourhood of $x$. We claim that $d_H \big( \Phi (x) , \Phi (y) \big) < \varepsilon$ for every $y \in U$. Let $y \in U$ be arbitrary. In order to show that $d_H \big( \Phi (x) , \Phi (y) \big) < \varepsilon$ it suffices to prove that for every $p \in \Phi (x)$ there is $q \in \Phi (y)$ with $\norm{p-q} < \varepsilon$ and for every $q \in \Phi (y)$ there is $p \in \Phi (x)$ with $\norm{p-q} < \varepsilon$. Let $p \in \Phi (x)$ be given and denote
\[R:= \bigcup_{n=1}^{\infty} \bigcup_{k \in {\mathcal{J}}_n} g_k^n \big( \big[ 0 , 2^{-n} \big] \big).\]
If $p \in P \cup R$, we take $q:=p$. Then $q \in P \cup R \subseteq \Phi (y)$ and $\norm{p-q} = 0 < \varepsilon$.\\
If $p = f_k^n (x_n,t)$ for some $n \leq n_0$, $k \in {\mathcal{I}}_n$ and $t \in \big[ 0 , 2^{-n} \big]$, let $q:=f_k^n (y_n,t)$. Then $q \in \Phi (y)$ and (using that $y \in U$ and $n \leq n_0$) we have
\[\norm{p-q} = \norm{f_k^n (x_n,t) - f_k^n (y_n,t)} = \norm{(0,x_n-y_n,0)} = |x_n-y_n| < \varepsilon .\]
Finally, if $p = f_k^n (x_n,t)$ for some $n > n_0$, $k \in {\mathcal{I}}_n$ and $t \in \big[ 0 , 2^{-n} \big]$, we define $q:=f_k^n (x_n,0)$. Then
\[q = \big( 2^{-n} , x_n , 0 \big) \in \big[ 0 , \tfrac{1}{4} \big] \times [1,2] \times \lbrace 0 \rbrace \subseteq P \subseteq \Phi (y)\]
and (since $n > n_0$)
\[\norm{p-q} = \norm{f_k^n (x_n,t) - f_k^n (x_n,0)} \leq \mathrm{diam} \Big( f_k^n \big( \lbrace x_n \rbrace \times \big[ 0 , 2^{-n} \big] \big) \Big) < \varepsilon .\]
A symmetrical argument shows that for every $q \in \Phi (y)$ there is $p \in \Phi (x)$ with $\norm{p-q} < \varepsilon$. 
\claimend
\end{claimproof}
\end{proof}

Before moving onto the last theorem of this chapter, let us present a few more lemmata.
\begin{lemma}\label{LocalConnectednessOfUnion}
Let $X$ be a compact metric space and $Y \subseteq X$ a locally connected continuum. Let $\mathcal{A} \subseteq \mathcal{P}(X)$ be a family such that:
\begin{itemize}
    \item every $A \in \mathcal{A}$ intersects Y;
    \item every $A \in \mathcal{A}$ is a locally connected continuum;
    \item the set $\big\lbrace A \in \mathcal{A} \, ; \, \mathrm{diam}(A) \geq \varepsilon \big\rbrace$ is finite for every $\varepsilon > 0$;
    \item $Y \cup \bigcup \mathcal{A} = X$.
\end{itemize}
Then $X$ is a locally connected continuum.
\begin{proof}
It is clear that $X$ is a continuum, we just need to prove that $X$ is locally connected. By \cite[Theorem 8.4]{Nadler}, it suffices to show that for every $\varepsilon > 0$ there is $n \in \N$ and continua $C_1 , \dotsc , C_n \subseteq X$ covering $X$ such that $\mathrm{diam} (C_i) < \varepsilon$ for every $i \in \{ 1 , \dotsc , n \}$. Let $\varepsilon > 0$ be arbitrary. Denote
${\mathcal{A}}_0 := \big\lbrace A \in \mathcal{A} \, ; \, \mathrm{diam} (A) \geq \frac{1}{3} \varepsilon \big\rbrace$ and $X_0 := Y \cup \bigcup {\mathcal{A}}_0$. Then ${\mathcal{A}}_0$ is finite and $X_0$ is a continuum. Moreover, as the union of finitely many locally connected closed sets, $X_0$ is locally connected. Hence, by \cite[Theorem 8.4]{Nadler}, there is $n \in \N$ and continua $K_1 , \dotsc , K_n \subseteq X_0$ covering $X_0$ such that $\mathrm{diam} (K_i) < \frac{1}{3} \varepsilon$ for every $i \in \{ 1 , \dotsc , n \}$. Denote $\Lambda := \{ 1 , \dotsc , n \}$. For every $i \in \Lambda$ let ${\mathcal{A}}_i := \big\lbrace A \in \mathcal{A} \setminus {\mathcal{A}}_0 \, ; \, A \cap K_i \neq \emptyset \big\rbrace$ and $C_i := K_i \cup \bigcup {\mathcal{A}}_i$. Clearly, ${\mathcal{A}}_1 \cup \dotsc \cup {\mathcal{A}}_n = \mathcal{A} \setminus {\mathcal{A}}_0$ and $C_1 \cup \dotsc \cup C_n = X_0 \cup \bigcup (\mathcal{A} \setminus {\mathcal{A}}_0) = Y \cup \bigcup \mathcal{A} = X$. Let us show that each of the sets $C_1 , \dotsc , C_n$ is a continuum with diameter less than $\varepsilon$. Let $i \in \Lambda$ be arbitrary. It is evident that $C_i$ is connected. Also, it is easy to prove (e.g. using Lemma \ref{ClosednessOfUnion}) that $C_i$ is closed in $X$. Thus, $C_i$ is compact. Finally, by the triangle inequality, $\mathrm{diam} (C_i) < \frac{\varepsilon}{3}+\frac{\varepsilon}{3}+\frac{\varepsilon}{3} = \varepsilon$.
\end{proof}
\end{lemma}
\begin{lemma}\label{MiniLemma}
Let $P$ be a connected topological space and $A \subseteq P$ an open set with connected boundary. Then $P \setminus A$ is connected.
\end{lemma}
\begin{proof}
Let $C_1 , C_2 \subseteq P \setminus A$ be closed sets (in $P \setminus A$ and thus also in $P$) such that $C_1 \cup C_2 = P \setminus A$ and $C_1 \cap C_2 = \emptyset$. Since $\partial A$ is a connected subset of $P \setminus A$, we can assume without loss of generality that $\partial A \subseteq C_1$. It follows that $A \cup C_1 = \overline{A} \cup C_1$, therefore $A \cup C_1$ is a closed subset of $P$. Finally, as $(A \cup C_1) \cap C_2 = \emptyset$ and $(A \cup C_1) \cup C_2 = P$, we conclude by the connectedness of $P$ that $C_2$ is empty.
\end{proof}
The following lemma is an easy consequence of Lemma \ref{MiniLemma}.

\begin{lemma}\label{PathCon}
Let $T$ be a topological space, $P \subseteq T$ a connected set and $\mathcal{A}$ a finite family of open subsets of $T$ such that:
\begin{enumerate}[label=\textup{(\roman*)}]

\item $\forall \, A \in \mathcal{A} : \ \overline{A} \subseteq P$;

\item $\forall \, A , B \in \mathcal{A} : \ A \neq B \implies \overline{A} \cap \overline{B} = \emptyset$;

\item $\partial A$ is connected for every $A \in \mathcal{A}$.

\end{enumerate}
Then $P \setminus \bigcup \mathcal{A}$ is a connected set.
\end{lemma}

\begin{lemma}\label{GisInsBdG}
Let $G \subseteq \R^2$ be an open bounded set whose boundary is a Jordan curve. Then $G = \ins{ ( \partial G ) }$.
\end{lemma}
\begin{proof}
Let us start by showing that $G$ is disjoint from $\out{( \partial G )}$. Suppose, for a contradiction, that $G \cap \out{( \partial G )} \neq \emptyset$. Since $G$ is bounded and $\out{( \partial G )}$ is unbounded, we have $\out{( \partial G )} \setminus \overline{G} \neq \emptyset$. As $\partial G \cap \out{( \partial G )} = \emptyset$, we receive
\[\out{( \partial G )} = \big( G \cap \out{( \partial G )} \big) \cup \big( \out{( \partial G )} \setminus \overline{G} \big) .\]
Thus, $\out{( \partial G )}$ can be expressed as the union of two nonempty disjoint open sets, which contradicts the fact that $\out{( \partial G )}$ is connected.

Knowing that $G \cap \out{( \partial G )} = \emptyset$, it follows from the openness of $G$ that $G \subseteq \ins{( \partial G )}$. Moreover, since $\partial G \cap \ins{( \partial G )} = \emptyset$, the set $G$ is relatively clopen in $\ins{( \partial G )}$. Thus, by the connectedness of $\ins{( \partial G )}$ and the fact that $G$ is nonempty (which is a consequence of $\partial G$ being a Jordan curve), we finally obtain the equality $G = \ins{( \partial G )}$.
\end{proof}

The following result was originally proved by Whyburn \cite{WhyburnSierpinski}.

\begin{lemma}\label{Carpet1}
Let $X \subseteq \R^2$ be a continuum and denote by $\, \mathcal{U}$ the family of all connected components of $\R^2 \setminus X$. Assume that:
\begin{enumerate}[label=\textup{(\arabic*)},noitemsep]

\item $\, \mathcal{U}$ is infinite;

\item $\forall \ U , V \in \mathcal{U} : \ U \neq V \implies \overline{U} \cap \overline{V} = \emptyset$;

\item the set $\partial U$ is a Jordan curve for every $U \in \mathcal{U}$;

\item the set $\bigcup \{ \partial U ; \ U \in \mathcal{U} \}$ is dense in $X$;

\item the set $\{ U \in \mathcal{U} ; \ \mathrm{diam} (U) \geq \varepsilon \}$ is finite for every $\varepsilon > 0$.
\end{enumerate}
Then $X$ is homeomorphic to the Sierpi\'nski carpet.
\end{lemma}

The following lemma is a consequence of the previous one and it will be more useful for us.

\begin{lemma}\label{Carpet2}
Let $J \subseteq \R^2$ be a Jordan curve and $\, \mathcal{V}$ an infinite family of open subsets of $\R^2$ such that:
\begin{enumerate}[label=\textup{(\roman*)},noitemsep]

\item $\forall \ V \in \mathcal{V} : \ \overline{V} \subseteq \ins{(J)}$;

\item $\forall \ U , V \in \mathcal{V} : \ U \neq V \implies \overline{U} \cap \overline{V} = \emptyset$;

\item the set $\partial V$ is a Jordan curve for every $V \in \mathcal{V}$;

\item the set $\bigcup \mathcal{V}$ is dense in $J \cup \ins{(J)}$;

\item the set $\{ V \in \mathcal{V} \, ; \ \mathrm{diam} (V) \geq \varepsilon \}$ is finite for every $\varepsilon > 0$.

\end{enumerate}

Then $\big( J \cup \ins{(J)} \big) \setminus \bigcup \mathcal{V}$ is homeomorphic to the Sierpi\'nski carpet.
\end{lemma}

\begin{proof}
Let $\mathcal{U}:=\mathcal{V} \cup \{ \out{(J)} \}$ and $X:= \big( J \cup \ins{(J)} \big) \setminus \bigcup \mathcal{V}$. We need to verify the assumptions of Lemma~\ref{Carpet1}. Clearly, $X$ is a compact subset of $\R^2$. Let us prove that $X$ is a continuum: As a disjoint family of open subsets of a separable space, $\mathcal{V}$ is countable. Therefore, we can write $\mathcal{V} = \{ V_k \, ; \, k \in \N \}$. For every $n \in \N$ let
\[ X_n := \big( J \cup \ins{(J)} \big) \setminus \bigcup_{k=1}^n V_k. \]
Trivially, we have $X_{n+1} \subseteq X_n$ for every $n \in \N$. Also, the intersection of the family $\{ X_n \, ; \, n \in \N \}$ is equal to $X$. Thus, it suffices to show that $X_n$ is a continuum for each $n \in \N$. Given an arbitrary $n \in \N$, denote $T := \R^2$, $P:=J \cup \ins (J)$ and $\mathcal{A} := \{ V_k \, ; \, k \leq n \}$. By Lemma \ref{PathCon} (whose assumptions can be easily verified), the set $P \setminus \bigcup \mathcal{A} = X_n$ is connected. Hence, as $X_n$ is clearly compact, we are done.

Now let us prove that $\mathcal{U}$ is the family of all connected components of $\R^2 \setminus X$. By (i), (iii) and Lemma~\ref{GisInsBdG}, we have $V = \ins{(\partial V)}$ for every $V \in \mathcal{V}$. In particular, every $V \in \mathcal{V}$ is connected. Hence, recalling the definition of $X$ and using the fact that $\out{(J)}$ is connected, we see that every member of $\mathcal{U}$ is a connected subset of $\R^2 \setminus X$. Since $\partial (\out{(J)}) = J \subseteq X$, it is clear that $\out{(J)}$ is a connected component of $\R^2 \setminus X$. Given $V \in \mathcal{V}$, let $C$ be the connected component of $\R^2 \setminus X$ containing $V$. By (i), (ii) and the definition of $X$, we have $\partial V \subseteq X$. Hence, $C \cap \partial V = \emptyset$ and therefore $C \setminus \out (\partial V) = C \cap \ins (\partial V) = C \cap V = V$, which shows that $V$ is relatively clopen in $C$. As $V$ is nonempty (by (iii)), it follows that $V = C$. Thus, we conclude that every $U \in \mathcal{U}$ is a connected component of $\R^2 \setminus X$. Moreover,
\[\bigcup \mathcal{U} = (\out{(J)}) \cup \bigcup \mathcal{V} = (\out{(J)}) \cup \big( (J \cup \ins{(J)}) \setminus X \big) = \R^2 \setminus X ,\]
hence $\mathcal{U}$ is indeed the family of all connected components of $\R^2 \setminus X$.

It remains to verify the assumptions (1) through (5) of Lemma~\ref{Carpet1}. As $\mathcal{V}$ is infinite, so is $\mathcal{U}$. Using (i), (ii) and the fact that $\overline{\out{(J)}} = \R^2 \setminus \ins{(J)}$, we conclude that $\overline{U} \cap \overline{V} = \emptyset$ for all $U , V \in \mathcal{U}$ with $U \neq V$. By (iii) and the fact that $\partial (\out{(J)}) = J$, we have that $\partial V$ is a Jordan curve for every $V \in \mathcal{U}$. It immediately follows from (v) and the finiteness of $\mathcal{U} \setminus \mathcal{V}$ that the set $\{ U \in \mathcal{U} \, ; \, \mathrm{diam} (U) \geq \varepsilon \}$ is finite for every $\varepsilon > 0$.

The only thing missing is the density of the set $\bigcup \{ \partial U ; \ U \in \mathcal{U} \}$ in $X$. Let $x \in X$ and $\varepsilon > 0$ be given. We need to show that there is a set $U \in \mathcal{U}$ and a point $y \in \partial U$ such that $\norm{x-y} < \varepsilon$. Let us assume that $x \notin \partial U$ for every $U \in \mathcal{U}$, otherwise we are immediately done. By (iv), the set $\bigcup \mathcal{V}$ is dense in $J \cup \ins{(J)}$, hence (as $x \in X$ and $X \subseteq J \cup \ins{(J)}$) there is a set $U \in \mathcal{V} \subseteq \mathcal{U}$ and a point $z \in U$ such that $\norm{x-z} < \varepsilon$. Consider the line segment $L:= \{ tx+(1-t)z \, ; \ t \in [0,1] \}$. It is a set containing both $z$ and $x$, hence it intersects both $U$ and $\R^2 \setminus \overline{U}$. Therefore, since $L$ is connected, it follows that $L \cap \partial U \neq \emptyset$. Choose an arbitrary point $y \in L \cap \partial U$. Then, evidently, $\norm{x-y} < \norm{x-z} < \varepsilon$.
\end{proof}

\begin{theorem}\label{thmLC1}
The homeomorphism equivalence relation of $1$-dimensional locally connected continua in $\R^2$ is not classifiable by countable structures.
\end{theorem}

\begin{proof}
As in the proof of Theorem~\ref{thmAR3}, consider the equivalence relation $E$ on ${[1,2]}^{\N}$ given by $\, x E y \iff \lim\limits_{n\to \infty} |x_n-y_n|=0$. Again, it suffices to show that $E$ is Borel reducible to the homeomorphism ER on ${\mathsf{LC}}_1 ( \R^2 )$.\\
Define a mapping $f \colon \R^8 \to \R^2$ by
\[f ( a_1 , a_2 , b_1 , b_2 , \widetilde{a_2} , \widetilde{b_2} , s , t ) = (st-s-t+1) \cdot (a_1 , a_2) + (s-st) \cdot (b_1 , b_2)\]
\[+ (t-st) \cdot (a_1 , \widetilde{a_2}) + st \cdot ( b_1 , \widetilde{b_2} ) .\]
Clearly, $f$ is smooth, hence it is Lipschitz on every compact subset of $\R^8$. In particular, there is $L > 0$ such that $f$ is $L$-Lipschitz on ${[0,3]}^8$.

Let $A := \big\lbrace ( a_1 , a_2 , b_1 , b_2 , \widetilde{a_2} , \widetilde{b_2} ) \in \R^6 \, ; \ a_1 < b_1 \, , \ a_2 < \widetilde{a_2} \, , \ b_2 < \widetilde{b_2} \big\rbrace$. For the sake of notational simplicity, let us identify $\R^8$ with $\R^6 \times \R^2$. For every $z \in A$ define a mapping $f_z \colon [0,1]^2 \to \R^2$ by $f_z (p) = f(z,p) \, , \ p \in [0,1]^2$. It is straightforward to show that for every $z = ( a_1 , a_2 , b_1 , b_2 , \widetilde{a_2} , \widetilde{b_2} ) \in A$ the mapping $f_z$ is a homeomorphism between $[0,1]^2$ and the convex hull of
\[\big\lbrace (a_1,a_2) , (b_1,b_2) , (a_1, \widetilde{a_2}) , ( b_1 , \widetilde{b_2}) \big\rbrace .\]
From now on, we will be using this property of the mappings $f_z$, $z \in A$, repeatedly without mentioning it, so bear that in mind when reading the rest of this proof.

Recalling (and slightly modifying) the classical iterative construction of the Sierpi\'nski Carpet, it is easy to see that for every $n \in \N$ it is possible to construct an infinite family ${\mathcal{W}}_n$ of open subsets of $\R^2$ such that:
\begin{itemize}
    \item $\forall \ W \in {\mathcal{W}}_n : \ \overline{W} \subseteq (0,1)^2$;
    \item $\forall \ W_1 , W_2 \in {\mathcal{W}}_n : \ W_1 \neq W_2 \implies \overline{W_1} \cap \overline{W_2} = \emptyset ;$
    \item the set $\partial W$ is a Jordan curve for every $W \in {\mathcal{W}}_n ;$
    \item the set $\bigcup {\mathcal{W}}_n$ is dense in $[0,1]^2$;
    \item the set $\{ W \in {\mathcal{W}}_n \, ; \ \mathrm{diam} (W) \geq \varepsilon \}$ is finite for every $\varepsilon > 0 ;$
    \item $\mathrm{diam} (W) < \frac{1}{n}$ for every $W \in {\mathcal{W}}_n$.
\end{itemize}
For every $n \in \N$ denote $c_n := 2^{-n}$ and $d_n := 2^{-n}+2^{-n-1}$. In addition, for every $n \in \N$ and every $x \in [1,2]^{\N}$ let ${\alpha}_n (x) := x_n - 2^{-n}$, ${\beta}_n (x) := x_n + 2^{-n}$.\\
Denote ${\N}_2 := \{ n \in \N \, ; \ n \geq 2 \}$ and define mappings\\
$m_1 , m_2 , m_3 , m_4 , m_5 \colon {\N}_2 \times [1,2]^{\N} \to A$ by
\begin{align*}
m_1 (n,x) &= \big( c_n , \, 0 , \, d_n , \, 0 , \, {\alpha}_n (x) , \, {\alpha}_n (x) \big) ,\\
m_2 (n,x) &= \big( c_n , \, {\beta}_n (x) , \, d_n , \, {\beta}_n (x) , \, 3 , \, 3 \big) ,\\
m_3 (n,x) &= \big( d_n , \, 0 , \, c_{n-1} , \, 0 , \, {\alpha}_n (x) , \, {\alpha}_{n-1} (x) \big) ,\\
m_4 (n,x) &= \big( d_n , \, {\beta}_n (x) , \, c_{n-1} , \, {\beta}_{n-1} (x) , \, 3 , \, 3 \big) ,\\
m_5 (n,x) &= \big( d_n , \, {\alpha}_n (x) , \, c_{n-1} , \, {\alpha}_{n-1} (x) , \, {\beta}_n (x) , \, {\beta}_{n-1} (x) \big) .\end{align*}
For every $x \in [1,2]^{\N}$ let
\[{\mathcal{V}}_0 (x) := \Big\lbrace (c_n,d_n) \times \big( {\alpha}_n (x) , {\beta}_n (x) \big) \, ; \ n \in {\N}_2 \Big\rbrace .\]
Moreover, for every $k \in \{ 1,2,3,4,5 \}$ and every $x \in [1,2]^{\N}$ let
\[{\mathcal{V}}_k (x) := \Big\lbrace f_{m_k (n,x)} (W) \, ; \ n \in {\N}_2 \, , \ W \in {\mathcal{W}}_n \Big\rbrace .\]
Finally, for every $x \in [1,2]^{\N}$ let
\[\mathcal{V} (x) := {\mathcal{V}}_0 (x) \cup {\mathcal{V}}_1 (x) \cup {\mathcal{V}}_2 (x) \cup {\mathcal{V}}_3 (x) \cup {\mathcal{V}}_4 (x) \cup {\mathcal{V}}_5 (x) .\]
Define a mapping ${\Phi}_0 \colon [1,2]^{\N} \to {\mathsf{LC}}_1 ( \R^2 )$ by
\[{\Phi}_0 (x) = \Big( \big[0, \tfrac{1}{2} \big] \times [0,3] \Big) \setminus \bigcup \mathcal{V} (x) .\]
To see that ${\Phi}_0$ is well-defined we prove the following.

\begin{claim}
The set ${\Phi}_0 (x)$ is homeomorphic to the Sierpi\'nski carpet (which is known to be a 1-dimensional locally connected continuum) for every point $x \in [1,2]^{\N}$.
\end{claim}

\begin{claimproof}
We shall use Lemma~\ref{Carpet2}. Let $x \in [1,2]^{\N}$ be given. Define
\[J:= \partial \Big( \big[ 0 , \tfrac{1}{2} \big] \times [0,3] \Big) = \Big( \big[ 0 , \tfrac{1}{2} \big] \times [0,3] \Big) \setminus \Big( \big( 0 , \tfrac{1}{2} \big) \times (0,3) \Big) .\]
Then $J$ is a Jordan curve and $J \cup \ins{(J)} = \big[ 0 , \tfrac{1}{2} \big] \times [0,3]$, hence
\[{\Phi}_0 (x) = \big( J \cup \ins{(J)} \big) \setminus \bigcup \mathcal{V} (x).\]
Clearly, $\mathcal{V} (x)$ is an infinite family of subsets of $\R^2$. It is also clear that each member of ${\mathcal{V}}_0 (x)$ is an open subset of $\R^2$. Moreover, since $f_{m_k (n,x)}$ is a homeomorphism between $[0,1]^2$ and $f_{m_k (n,x)} \big( [0,1]^2 \big) \subseteq \R^2$ for every $k \in \{ 1,2,3,4,5 \}$ and $n \in {\N}_2$, we conclude (recalling that ${\mathcal{W}}_n$, $n \in \N$, are families of open subsets of $(0,1)^2$ and using the domain invariance theorem) that each member of $\mathcal{V} (x)$ is an open subset of $\R^2$.

Evidently, we have $\overline{V} \subseteq \big( 0 , \tfrac{1}{2} \big) \times (0,3) = \ins{(J)}$ for every $V \in {\mathcal{V}}_0 (x)$. Moreover, for every $k \in \{ 1 , 2 , 3 , 4 , 5 \}$ and $n \in {\N}_2$, the set $f_{m_k (n,x)} \big( [0,1]^2 \big)$ is a subset of $\big[ 0 , \tfrac{1}{2} \big] \times [0,3]$. Hence, using that $\overline{W} \subseteq (0,1)^2$ for every $W \in {\mathcal{W}}_n$, $n \in \N$, it follows that
\[\overline{f_{m_k (n,x)} (W)} = f_{m_k (n,x)} ( \overline{W} ) \subseteq f_{m_k (n,x)} \big( (0,1)^2 \big) \subseteq \big( 0 , \tfrac{1}{2} \big) \times (0,3) = \ins{(J)} ,\]
$k \in \{ 1 , 2 , 3 , 4 , 5 \}$, $n \in {\N}_2$, $W \in {\mathcal{W}}_n$. Thus, $\overline{V} \subseteq \ins{(J)}$ for all $V \in \mathcal{V} (x)$.

Now let us explain why $\overline{U} \cap \overline{V} = \emptyset$ for all $U , V \in \mathcal{V} (x)$ such that $U \neq V$. For notational simplicity, denote $\Lambda := \{ 1 , 2 , 3 , 4 , 5 \}$ and $O := (0,1)^2$. Since $\overline{W_1} \cap \overline{W_2} = \emptyset$ for all $W_1 , W_2 \in {\mathcal{W}}_n$ with $W_1 \neq W_2$ and all $n \in \N$, we have
\[\overline{f_{m_{k} (n,x)} (W_1)} \cap \overline{f_{m_{k} (n,x)} (W_2)} = f_{m_{k} (n,x)} (\overline{W_1}) \cap f_{m_{k} (n,x)} (\overline{W_2}) \]
\[= f_{m_{k} (n,x)} ( \overline{W_1} \cap \overline{W_2} ) = f_{m_{k} (n,x)} (\emptyset) = \emptyset\]
for all $k \in \Lambda$, $n \in {\N}_2$ and $W_1 , W_2 \in {\mathcal{W}}_n$ with $W_1 \neq W_2$. Moreover, it is straightforward to verify that
\[\forall \ U , V \in {\mathcal{V}}_0 (x) : \ U \neq V \implies \overline{U} \cap \overline{V} = \emptyset \, ,\]
\[\forall \, V \in {\mathcal{V}}_0 (x) \, \forall \, k \in \Lambda \ \forall \, n \in {\N}_2 : V \cap f_{m_k (n,x)} (O) = \emptyset \, ,\]
\[\forall \, k , l \in \Lambda \ \forall \, n , \widehat{n} \in {\N}_2 : (k,n) \neq ( l , \widehat{n} ) \implies f_{m_{k} (n,x)} (O) \cap f_{m_{l} ( \widehat{n} , x)} (O) = \emptyset \, .\]
Therefore, as $\overline{f_{m_k (n,x)} (W)} = f_{m_k (n,x)} (\overline{W}) \subseteq f_{m_k (n,x)} (O)$ for all $k \in \Lambda$, $n \in {\N}_2$ and $W \in {\mathcal{W}}_n$, we can indeed conclude that $\overline{U} \cap \overline{V} = \emptyset$ for all $U , V \in \mathcal{V} (x)$ with $U \neq V$.

Clearly, $\partial V$ is a Jordan curve for every $V \in {\mathcal{V}}_0 (x)$. Moreover, for every $k \in \Lambda$, $n \in {\N}_2$ and $W \in {\mathcal{W}}_n$ we have $\partial \big( f_{m_k (n,x)} (W) \big) = f_{m_k (n,x)} (\partial W)$, thus $\partial \big( f_{m_k (n,x)} (W) \big)$ is homeomorphic to $\partial W$, which is a Jordan curve. Hence, $\partial V$ is a Jordan curve for every $V \in \mathcal{V} (x)$.

Now let us prove that $\bigcup \mathcal{V} (x)$ is dense in $J \cup \ins{(J)}$. It is easy to see that
\[\Big( \bigcup {\mathcal{V}}_0 (x) \Big) \cup \bigcup_{k \in \Lambda} \bigcup_{n \in {\N}_2} f_{m_k (n,x)} \big( [0,1]^2 \big) = \big( 0 , \tfrac{1}{2} \big] \times [0,3].\]
Therefore, as $\big( 0 , \tfrac{1}{2} \big] \times [0,3]$ is dense in $\big[ 0 , \tfrac{1}{2} \big] \times [0,3] = J \cup \ins{(J)}$, it suffices to show that $ f_{m_k (n,x)} \big( \bigcup {\mathcal{W}}_n \big)$ is dense in $f_{m_k (n,x)} \big( [0,1]^2 \big)$ for every $k \in \Lambda$ and $n \in {\N}_2$. However, we have
\[\overline{f_{m_k (n,x)} \Big( \bigcup {\mathcal{W}}_n \Big) } = f_{m_k (n,x)} \Big( \overline{\bigcup {\mathcal{W}}_n} \Big) = f_{m_k (n,x)} \big( [0,1]^2 \big)\]
for every $k \in \Lambda$ and $n \in {\N}_2$. Thus, we are done.

Finally, let us show that the set $\{ V \in \mathcal{V} (x) \, ; \, \mathrm{diam} (V) \geq \varepsilon \}$ is finite for every $\varepsilon > 0$. For every $n \in {\N}_2$ we have
\[\mathrm{diam} \Big( (c_n , d_n) \times \big( {\alpha}_n (x) , {\beta}_n (x) \big) \Big) = \norm{\big( c_n - d_n \, , \ {\alpha}_n (x) - {\beta}_n (x) \big) }\]
\[\leq |c_n - d_n| + |{\alpha}_n (x) - {\beta}_n (x)| = 2^{-1-n}+2^{1-n} < 2^{2-n} .\]
In particular, $\{ V \in {\mathcal{V}}_0 (x) \, ; \, \mathrm{diam} (V) \geq \varepsilon \}$ is a finite set for every $\varepsilon > 0$. Now, let $k \in \Lambda$ and $\varepsilon > 0$ be arbitrary. We have to show that the set $\{ V \in {\mathcal{V}}_k (x) \, ; \, \mathrm{diam} (V) \geq \varepsilon \}$ is finite. It is clear that $m_k (n,x) \in [0,3]^6$ for every $n \in {\N}_2$. Moreover, every member of each of the families ${\mathcal{W}}_n$, $n \in \N$, is a subset of $(0,1)^2$. Therefore, since the mapping $f$ is $L$-Lipschitz on $[0,3]^8$ and $f_{m_k (n,x)} (w) = f \big( m_k (n,x) , w \big)$ for all $n \in {\N}_2$, $W \in {\mathcal{W}}_n$ and $w \in W$, it follows that
\[\mathrm{diam} \big( f_{m_k (n,x)} (W) \big) \leq L \cdot \mathrm{diam} (W) \, , \ \ n \in {\N}_2 \, , \, W \in {\mathcal{W}}_n \, .\]
Thus, for all $n \in {\N}_2$ we see that each $W \in {\mathcal{W}}_n$ with $\mathrm{diam} \big( f_{m_k (n,x)} (W) \big) \geq \varepsilon$ belongs to the set $\big\lbrace \widetilde{W} \in {\mathcal{W}}_n \, ; \ \mathrm{diam} ( \widetilde{W} ) \geq \frac{\varepsilon}{L} \big\rbrace$, which is finite. Hence, since $\mathrm{diam} (W) < \frac{1}{n} \leq \frac{\varepsilon}{L}$ for every $n \in \N$ satisfying $n \geq \frac{L}{\varepsilon}$ and every $W \in {\mathcal{W}}_n$, it indeed follows that the set $\{ V \in {\mathcal{V}}_k (x) \, ; \, \mathrm{diam} (V) \geq \varepsilon \}$ is finite.

Having just verified the assumptions of Lemma~\ref{Carpet2}, we know that ${\Phi}_0 (x)$ is homeomorphic to the Sierpi\'nski carpet.
\claimend
\end{claimproof}

Now let us modify ${\Phi}_0$ to make it a reduction from $E$ to the homeomorphism ER on ${\mathsf{LC}}_1 ( \R^2 )$. Let $(q_i)_{i=1}^{\infty}$ be an injective sequence in $[0,3]$ such that the set $\{ q_i \, ; \, i \in \N \}$ is dense in $[0,3]$. For every $i \in \N$ let
\[r_i := \frac{1}{3} \mathrm{min} \Big( \{ 2^{-i} \} \cup \big\lbrace |q_i - q_j| \, ; \, j < i \big\rbrace \Big) .\]
Then $r_i > 0$ for every $i \in \N$ and the sequence $(r_i)_{i=1}^{\infty}$ tends to $0$. For every $n \in \N$ denote ${\mathcal{J}}_n := \{ k \in \mathbb{Z} \, ; -n \leq k \leq n \}$ and ${\mathcal{I}}_n := {\mathcal{J}}_n \setminus \{ 0 \}$. For every $n \in \N$ and $k \in {\mathcal{J}}_n$ define a mapping $g_k^n \colon \R \to \R^2$ by
\[g_k^n (t) = \Big(-t r_n \, , \ q_n + \frac{tk}{n} r_n \Big) .\]
For every $n \in \N$, let $\displaystyle M_n := \bigcup_{k \in {\mathcal{J}}_n} g_k^n \big( [0,1] \big)$. Moreover, define $\displaystyle M := \bigcup_{n=1}^{\infty} M_n$.

It is easy to verify that the sets $M_n$, $n \in \N$, are pairwise disjoint and each of them belongs to ${\mathsf{LC}}_1 ( \R^2 )$. For every $n \in \N$ and $k \in {\mathcal{I}}_n$ define a mapping $h_k^n \colon [1,2] \times \R \to \R^2$ by
\[h_k^n (z,t) = \bigg( 2^{-n} \Big( 1+\frac{t}{4} \Big) \, , \, z + \frac{tk}{2^{n+1} n} \bigg).\]
Then $h_k^n \big( \{ z \} \times [0,1] \big) \in {\mathsf{LC}}_1 ( \R^2 )$ for all $n \in \N$, $k \in {\mathcal{I}}_n$ and $z \in [1,2]$. Also, we have $h_k^n \big( \{ x_n \} \times (0,1] \big) \subseteq (c_n , d_n) \times \big( {\alpha}_n (x) , {\beta}_n (x) \big)$ and $h_k^n (x_n , 0) = (c_n , x_n)$ for all $n \in \N$, $k \in {\mathcal{I}}_n$ and $x \in [1,2]^{\N}$. Define a mapping $\Phi \colon [1,2]^{\N} \to {\mathsf{LC}}_1 ( \R^2 )$ by
\[\Phi (x) = {\Phi}_0 (x) \cup M \cup \bigcup_{n=2}^{\infty} \bigcup_{k \in {\mathcal{I}}_n} h_k^n \big( \{ x_n \} \times [0,1] \big) .\]
Let us verify that $\Phi$ is well-defined.

\begin{claim}
The set $\Phi (x)$ belongs to ${\mathsf{LC}}_1 ( \R^2 )$ for every $x \in [1,2]^{\N}$. 
\end{claim}

\begin{claimproof}
Let $x \in [1,2]^{\N}$ be given. It is easy to verify (e.g. using Lemma~\ref{ClosednessOfUnion}) that $\Phi (x)$ is closed in $\R^2$. Hence, since it is clear that $\Phi (x)$ is bounded, $\Phi (x)$ is compact. Moreover, it is straightforward to check the assumptions of Lemma \ref{LocalConnectednessOfUnion} and thus conclude that $\Phi (x)$ is a locally connected continuum. Finally, as the union of countably many closed 1-dimensional sets, $\Phi (x)$ is 1-dimensional. \claimend
\end{claimproof}

Now let us show that $\Phi$ is a reduction from $E$ to the homeomorphism ER on ${\mathsf{LC}}_1 ( \R^2 )$. Let $x , y \in {\left[1 , 2 \right]}^{\N}$ be given.

\begin{claim}
If $xEy$ then $\Phi(x)$ is homeomorphic to $\Phi(y)$.
\end{claim}

\begin{claimproof}
Assume that $x E y$ and define a mapping $\varphi \colon \Phi (x) \to \Phi (y)$ by
\begin{align*}
\varphi (p) &= f_{m_k (n , y)} \big( f_{m_k (n , x)}^{-1} (p) \big) , \ p \in \Phi (x) \cap f_{m_k (n , x)} \big( [0 , 1]^2 \big) , \, k \in \Lambda , \, n \in {\N}_2\\
\varphi (p) &= p + (0 , y_n - x_n) , \ p \in \Phi (x) \cap \big( [c_n , d_n] \times [{\alpha}_n (x) , {\beta}_n (x)] \big) , \, n \in {\N}_2\\
\varphi (p) &= p , \ p \in M \cup \big( \{ 0 \} \times [0,3] \big) .\end{align*}
It is straightforward to verify that $\varphi$ is well-defined, injective and surjective. Hence, since $\Phi (x)$ is compact, it suffices to show that $\varphi$ is continuous. Let $p \in \Phi (x)$ be arbitrary, we need to show that $\varphi$ is continuous at $p$. We~will consider only the most complicated case, that is, $p \in \{ 0 \} \times [0,3]$. Let~$\varepsilon > 0$ be given. Recall that $f$ is $L$-Lipschitz on the set $[0,3]^8$. Since $xEy$, there is $n_0 \in \N$ such that $|x_n-y_n| < \mathrm{min} \big\lbrace \frac{1}{2} \varepsilon , \frac{1}{8L} \varepsilon \big\rbrace$ for every $n \in \N$ with $n \geq n_0$. Define $\delta := \mathrm{min} \big\lbrace \frac{1}{2} \varepsilon , \, 2^{-n_0} \big\rbrace$ and let $p' \in \Phi (x)$ be an arbitrary point satisfying $\norm{p-p'}< \delta$. We claim that $\norm{\varphi (p) - \varphi (p')} < \varepsilon$. Let us distinguish three cases: If $p' \in M \cup \big( \{ 0 \} \times [0,3] \big)$, we immediately receive
\[ \norm{\varphi (p) - \varphi (p')} = \norm{p-p'} < \delta \leq \tfrac{1}{2} \varepsilon < \varepsilon .\]
If $p'$ lies in $\Phi (x) \cap \big( [c_n , d_n] \times [{\alpha}_n (x) , {\beta}_n (x)] \big)$ for some $n \in {\N}_2$, we clearly have
\[2^{-n_0} \geq \delta > \norm{p-p'} \geq \mathrm{dist} \big( \{ 0 \} \times \R \, , \, [c_n , \infty ) \times \R \big)=c_n=2^{-n},\]
hence $n > n_0$, which implies that $|x_n-y_n| < \frac{1}{2} \varepsilon$. Therefore,
\begin{align*}
\norm{\varphi (p) - \varphi (p')} &= \norm{p-\big( p'+(0 , y_n - x_n) \big)} \leq \norm{p-p'}+\norm{(0 , x_n - y_n)}\\
&< \delta + |x_n-y_n| < \delta + \tfrac{1}{2} \varepsilon \leq \tfrac{1}{2} \varepsilon + \tfrac{1}{2} \varepsilon = \varepsilon.\end{align*}
Finally, if $p' \in \Phi (x) \cap f_{m_k (n , x)} \big( [0 , 1]^2 \big)$ for some $k \in \Lambda$ and $n \in {\N}_2$, there is a unique point $a \in [0 , 1]^2$ such that $f_{m_k (n , x)} (a) = p'$. By the definition of $\varphi$, we have
\begin{align*}
\norm{p'-\varphi (p')} &= \big\lVert p'-f_{m_k (n , y)} \big( f_{m_k (n , x)}^{-1} (p) \big) \big\rVert = \norm{p'-f_{m_k (n , y)} (a)}\\
&= \norm{f_{m_k (n , x)} (a)-f_{m_k (n , y)} (a)}\\
&= \big\lVert f ( m_k (n , x) , a ) - f ( m_k (n , y) , a ) \big\rVert \\
&\leq L \big\lVert ( m_k (n , x) , a ) - ( m_k (n , y) , a ) \big\rVert = L \big\lVert m_k (n , x) - m_k (n , y) \big\rVert . \end{align*}
Moreover, by the definition of the mappings $m_1$, $m_2$, $m_3$, $m_4$, $m_5$, we have
\begin{align*}
    m_1 (n , x) - m_1 (n , y) &= \big( 0 , \, 0 , \, 0 , \, 0 , \, {\alpha}_n (x) - {\alpha}_n (y) , \ {\alpha}_n (x) - {\alpha}_n (y) \big) ,\\
    m_2 (n , x) - m_2 (n , y) &= \big( 0 , \, {\beta}_n (x) - {\beta}_n (y) , \, 0 , \, {\beta}_n (x) - {\beta}_n (y) , \, 0 , \, 0 \big) ,\\
    m_3 (n , x) - m_3 (n , y) &= \big( 0 , \, 0 , \, 0 , \, 0 , \, {\alpha}_n (x) - {\alpha}_n (y) , \ {\alpha}_{n-1} (x) - {\alpha}_{n-1} (y) \big) ,\\
    m_4 (n , x) - m_4 (n , y) &= \big( 0 , \, {\beta}_n (x) - {\beta}_n (y) , \, 0 , \, {\beta}_{n-1} (x) - {\beta}_{n-1} (y) , \, 0 , \, 0 \big)
\end{align*}
and $m_5 (n , x) - m_5 (n , y)=$
\[\big( 0 , \, {\alpha}_n (x) - {\alpha}_n (y) , \, 0 , \, {\alpha}_{n-1} (x) - {\alpha}_{n-1} (y) , \ {\beta}_n (x) - {\beta}_n (y) , \ {\beta}_{n-1} (x) - {\beta}_{n-1} (y) \big) .\]
It follows that
\begin{align*}
    m_1 (n , x) - m_1 (n , y) &= \big( 0 , \, 0 , \, 0 , \, 0 , \, x_n - y_n , \ x_n - y_n \big) ,\\
    m_2 (n , x) - m_2 (n , y) &= \big( 0 , \, x_n - y_n , \, 0 , \, x_n - y_n , \, 0 , \, 0 \big) ,\\
    m_3 (n , x) - m_3 (n , y) &= \big( 0 , \, 0 , \, 0 , \, 0 , \, x_n - y_n , \ x_{n-1} - y_{n-1} \big) ,\\
    m_4 (n , x) - m_4 (n , y) &= \big( 0 , \, x_n - y_n , \, 0 , \, x_{n-1} - y_{n-1} , \, 0 , \, 0 \big) ,\\
    m_5 (n , x) - m_5 (n , y) &= \big( 0 , \, x_n - y_n , \, 0 , \, x_{n-1} - y_{n-1} , \, x_n - y_n , \, x_{n-1} - y_{n-1} \big) ,
\end{align*}
hence $\big\lVert m_l (n , x) - m_l (n , y) \big\rVert \leq 2 \, | x_n - y_n | + 2 \, | x_{n-1} - y_{n-1} |$ for every $l \in \Lambda$. In particular, $\big\lVert m_k (n , x) - m_k (n , y) \big\rVert \leq 2 \, | x_n - y_n | + 2 \, | x_{n-1} - y_{n-1} |$. Also, recalling the definition of the mappings $m_1$, $m_2$, $m_3$, $m_4$, $m_5$ and using the fact that for every $z = ( a_1 , a_2 , b_1 , b_2 , \widetilde{a_2} , \widetilde{b_2} ) \in A$ the mapping $f_z$ maps $[0,1]^2$ onto the convex hull of $\big\lbrace (a_1 , a_2) , (b_1 , b_2) , (a_1 , \widetilde{a_2}) , ( b_1 , \widetilde{b_2} ) \big\rbrace$, we easily conclude that $p' \in [c_n , \infty ) \times \R = [2^{-n} , \infty ) \times \R$. Thus,
\[2^{-n_0} \geq \delta > \norm{p-p'} \geq \mathrm{dist} \big( \{ 0 \} \times \R \, , \, [2^{-n} , \infty ) \times \R \big)=2^{-n}.\]
It follows that $n>n_0$ and $n-1 \geq n_0$, which gives us
\[| x_n - y_n | < \tfrac{1}{8L} \varepsilon \ , \ \ \  | x_{n-1} - y_{n-1} | < \tfrac{1}{8L} \varepsilon \, .\]
Therefore, we have $2 \, | x_n - y_n | + 2 \, | x_{n-1} - y_{n-1} | < \tfrac{2}{8L} \varepsilon + \tfrac{2}{8L} \varepsilon = \frac{1}{2L} \varepsilon$, thus $\big\lVert m_k (n , x) - m_k (n , y) \big\rVert < \frac{1}{2L} \varepsilon$, hence $\norm{p'-\varphi (p')} < L \cdot \frac{1}{2L} \varepsilon = \frac{1}{2} \varepsilon$. Finally,
\[\norm{\varphi (p) - \varphi (p')} = \norm{p - \varphi (p')} \leq \norm{p - p'} + \norm{p' - \varphi (p')} < \delta + \tfrac{1}{2} \varepsilon \leq \varepsilon .\]
\claimend
\end{claimproof}

\begin{claim}
If $\Phi(x)$ is homeomorphic to $\Phi(y)$ then $x E y$. 
\end{claim}

\begin{claimproof}
We proceed by contradiction, similarly to the corresponding part of the proof of Theorem~\ref{thmAR3}. Assume that $\Phi(x)$ is homeomorphic to $\Phi(y)$ and $(x,y) \notin E$. Let $\varphi \colon \Phi (x) \to \Phi (y)$ be an onto homeomorphism. Observe that for every $w \in {\left[1 , 2 \right]}^{\N}$ and $n \in \N$ the point $(0 , q_n)$ is the unique point $p \in \Phi (w)$ such that the set $\Phi (w) \setminus \{ p \}$ has exactly $2n+2$ connected components. Similarly, for every $w \in {\left[1 , 2 \right]}^{\N}$ and every $n \in {\N}_2$ the point $( 2^{-n} , w_n )$ is the unique point $p \in \Phi (w)$ such that the set $\Phi (w) \setminus \{ p \}$ has exactly $2n+1$ connected components. Using this observation, together with the fact that $\varphi$ is a homeomorphism and the set $\left\lbrace q_n ; n \in \N \right\rbrace$ is dense in $\left[0 , 3 \right]$, we conclude that $\varphi (p) = p$ for every $p \in \{ 0 \} \times [0 , 3]$ and $\varphi (2^{-n} , x_n) = (2^{-n} , y_n)$ for every $n \in {\N}_2$. Therefore, repeating the process from the proof of Theorem~\ref{thmAR3}, we easily deduce a contradiction.
\claimend
\end{claimproof}

It remains to show that $\Phi$ is Borel measurable. However, similarly as in the proof of Theorem~\ref{thmAR3}, it can be shown (although it is technically more complicated) that $\Phi$ is actually continuous.
\end{proof}

\bibliographystyle{alpha}
\bibliography{citace}
\end{document}